\documentclass{amsart}

\usepackage{graphicx, amssymb, diagram}

\theoremstyle{plain}
\newtheorem{theorem}{Theorem}[section]
\newtheorem{lemma}[theorem]{Lemma}
\newtheorem{corollary}[theorem]{Corollary}
\newtheorem*{intro_theorem}{Theorem}

\theoremstyle{definition}
\newtheorem{definition}[theorem]{Definition}

\theoremstyle{remark}
\newtheorem{remark}[theorem]{Remark}

\theoremstyle{example}

\begin{document}

\title{Duality for toric Landau-Ginzburg models}

\author{Patrick Clarke}
\address{Department of Mathematics, University of Pennsylvania,
Philadelphia, Pennsylvania 19104}
\curraddr{Department of Mathematics,
University of Pennsylvania, Philadelphia, Pennsylvania 19104}
\email{pclarke@math.upenn.edu}
\thanks{The author was supported in part by NSF Award \#0703643.}



\begin{abstract}
We introduce a duality construction for  toric
Landau-Ginzburg models, applicable to complete
intersections in toric varieties via the sigma model / Landau-Ginzburg
model correspondence.  This
construction is shown to reconstruct  those
of
Batyrev-Borisov, Berglund-H\"ubsch,
Givental, and Hori-Vafa. It
can be done in more general situations, and
provides partial resolutions when the above
constructions give a singular mirror. 
An extended  example is given:  the
Landau-Ginzburg models dual to
elliptic
curves
in $(\Bbb{P}^1)^2$.
\end{abstract}

\maketitle

\section*{introduction}

Motivated by the mirror involution on $N = 2$ superconformal theories 
and the fact that one can define from a Calabi-Yau such a theory,
string theorists believe that Calabi-Yaus exist in ``mirror pairs''.
In other words, to any Calabi-Yau, $Z$, there exists
another, $\check{Z}$, such that the superconformal theories 
they define are mirror to each other.

Mathematically, such a pair would have mirror symmetric Hodge diamonds,
\begin{equation}
h^{p,q}(Z) = h^{n-p, q}(\check{Z}) ,
\end{equation}
it would be possible to compute the Gromov-Witten invariants of $Z$ 
from the periods of $\check{Z}$, and, following the homological
mirror symmetry conjecture of Kontsevich \cite{kontsevich-icm}, the derived
category of coherent sheaves on $Z$, $D^b( coh \  Z)$, would be 
equivalent to the derived Fukaya category of $\check{Z}$, $DFuk(\check{Z})$.
Broadly speaking, mirror symmetry transforms
invariants of the symplectic topology of $Z$ into invariants of the
complex structure of $\check{Z}$, and vice versa.

For general K\"ahler manifolds $Z$ it is possible to define
Hodge numbers and $D^b( coh \ Z )$, and often (e.g. if $Z$ is Fano) 
Gromov-Witten invariants and $DFuk(Z)$ make sense.  This leads
naturally to the question of mirror pairs for general K\"ahler
manifolds, rather than just Calabi-Yaus.  One immediately
obvious hurdle is the existence of a K\"ahler manifold,
$\check{Z}$, satisfying the mirror symmetric Hodge
diamond relation above is only possible in the Calabi-Yau
case since 
\begin{equation}
h^{n, 0}(Z) = h^{0,0}(\check{Z}) = 1.
\end{equation}

A promising solution to problems such as this  is the use of Landau-Ginzburg 
models as mirrors.  A {\it Landau-Ginzburg model} is a superconformal theory
defined by  
a K\"ahler manifold 
$\check{X}$ equipped with a holomorphic function $\check{W} \colon \check{X} 
\rightarrow \Bbb{C}$.  
The function $\check{W}$ is referred to as the {\it superpotential}.
The Hodge numbers are then replaced by dimensions of graded 
components of a certain ``chiral'' ring associated with $Z$ or the pair 
$(\check{X},\check{W})$.
There also exists a version of the derived category of coherent sheaves
for Landau-Ginzburg models, $DB(\check{X}, \check{W})$, introduced by Orlov \cite{orlov-singDLG}
(generalizing the category of matrix factorizations), 
and a version of the derived Fukaya category, $DFS(\check{X}, \check{W})$, 
due to Seidel \cite{seidel-vanMut}.  

There are four general predictive methods 
for computing the mirror of a complete intersection in a toric variety.
The first to appear in the literature 
is that of Berglund and H\"ubsch 
\cite{BerglundHubsch:genConstruct}.  Their construction
produces a mirror candidate to a Calabi-Yau hypersurface in weighted projective
space.  Their mirror is also a Calabi-Yau hypersuface in (a different)
weighted projective space.

Shortly after this, Batyrev \cite{dualBatyrev} gave a construction for Calabi-Yau
hypersurfaces in Gorenstein toric Fano varieties.   This was subsequently
generalized by Borisov \cite{dualNefBorisov} to Calabi-Yau complete intersections that
arise from ``nef-partitions'' of the anti-canonical divisor.  Since
a weighted projective space is never Gorenstein unless it is projective
space itself, Berglund-H\"ubsch and Batyrev-Borisov address distinct
situations.

The combined efforts of Batyrev \cite{alg-geom/9711008}, Batyrev-Borisov 
\cite{alg-geom/9509009}, and 
Kontsevich \cite{kontsevich-motivicInt} ultimately led to the proof that $n$-dimensional 
Batyrev-Borsiov pairs, $Z$ and $\check{Z}$, have mirror symmetric 
(stringy) Hodge diamonds.  Thus giving one of the first general and rigorous 
results in mirror symmetry.

The same year, Givental published his ``mirror theorem'' \cite{mirrorTheoremGivental}
that made rigorous and generalized the approach to computing Gromov-Witten
invariants pioneered by physicists in \cite{CandelasDeLaOssaGreenParkes}.  Given a Fano, 
complete intersection $Z$ in a smooth projective toric variety, Givental 
constructs a Landau-Ginzburg model mirror candidate $(\check{X}, \check{W})$. 
He then shows that structure constants of the quantum cohomology of $Z$ can 
be found by considering certain
integrals over cycles in $\check{X}$ related to the Morse theory of $\check{W}$.
It is worth mentioning that the recipe given by Givental for the mirror
Landau-Ginzburg model can be done for arbitrary complete intersections
in toric varieties, even though he only considers the case of Fano
manifolds.

The most recent algorithm to compute a mirror candidate is given by 
Hori and Vafa \cite{HoriVafa}.  Using physical arguments, from a complete intersection in 
a smooth toric variety they obtain a mirror Landau-Ginzburg model.

This paper puts forth a new method for computing mirror candidates
for complete intersections in toric varieties.  Given
an $n$-dimensional 
toric variety $X$, an element $\omega \in$ the Chow group 
$A_{n-1}(X)_{\Bbb{C}/\Bbb{Z}}$ (coefficients in $\Bbb{C}/\Bbb{Z}$),
and a morphism $W \colon X \rightarrow \Bbb{C}$, we produce a
$n$-dimensional 
toric variety $X'$, a Chow group element $\omega' \in 
A_{n-1}(X')_{\Bbb{C}/\Bbb{Z}}$,
and a morphism $W' \colon  X' \rightarrow \Bbb{C}$
(strictly speaking the most natural objects to consider are 
toric Deligne-Mumford stacks, but we will not need this and the 
generalization is obvious).
We call the new Landau-Ginzburg model {\it dual} to the original.

Using an idea from physics called the {\it sigma model / Landau-Ginzburg 
model correspondence}, this process can be applied to generate a 
mirror candidate for a complete intersection in a toric variety.
This correspondence goes as follows.  Assume $Z$ is the zero locus
of a global section $\mathrm{w} \in \Gamma(Y, \mathcal{V})$ of some vector bundle
$\mathcal{V}$ over a K\"ahler manifold $Y$.  The identification
$\mathrm{Hom}(\mathcal{V}^\vee, \mathcal{O}_Y) \cong \Gamma(Y, \mathcal{V})$
allows one to use $\mathrm{w}$ to define a morphism $W \colon X \rightarrow \Bbb{C}$
on the total space $X = \mathrm{tot}(\mathcal{V}^\vee)$.
Physically, the superconformal theories defined by $Z$ and $(X,W)$
are the same \cite{SharpeGuffin:A-twistLG}.  Based on this, one would expect that
the Hodge numbers of $Z$ give the graded component dimensions
of the chiral ring of $(X,W)$, etc.
Landau-Ginzburg model mirror candidates are then formed by the composition:
\begin{equation}
Z \stackrel{\Sigma/LG}{\longmapsto} (X,W) \stackrel{dual}{\longmapsto} 
(X', W') .
\end{equation}
If $(X', W')$ 
has the form of a vector bundle paired with a section
of its dual, we can run the correspondence backwards,
 $(X', W') \stackrel{LG/\Sigma}{\longmapsto} 
Z'$, to obtain a K\"ahler manifold mirror candidate.

After initial definitions and describing the construction of
the dual Landau-Ginzburg model, we compare the mirror candidate 
obtained using four methods above with the candidate given by 
the dual.  Ultimately, all methods are shown to be
special cases of the duality defined here.

The Givental and Hori-Vafa mirrors are considered together.
By a simple computation, we show they both produce the same mirror.
The starting point for their construction is a split 
vector bundle $\mathcal{V}$ over toric variety $Y$.
The complete intersection is given by a global section of $\mathcal{V}$. 
Note that in their formulation, the complex structure of $Z$
is ignored.  
Concretely, this means that $Z$ is the zero locus
of some global section  $\mathrm{w} \in \Gamma(Y,\mathcal{V})$,
but the choice of $\mathrm{w}$ is not important.
This is justified by the fact that 
whenever integral, smooth subvarieties $Z_1$ and $Z_2$ are given by 
global sections $\mathrm{w}_1$
and $\mathrm{w}_2$ respectively, they are 
symplectomorphic.  With this in mind, we prove the following theorem.

\begin{intro_theorem}
For a specific choice $\mathrm{w}_{GHV} \in \Gamma(Y,\mathcal{V})$,
the mirror Landau-Ginzburg model of Givental-Hori-Vafa is
the dual Landau-Ginzburg model, $(X', W')$, 
to $(X = \mathrm{Tot}(\mathcal{V}^\vee), W_{GHV})$.
Where $W_{GHV}$ is defined by $\mathrm{w}_{GHV}$.
\end{intro_theorem}

It is nice to note that
the dual to $(X', W')$ is $(X,W^-)$,
where $X$ is the original space and $W^-$
is closely related to the original superpotential $W$.
This gives a nice resolution to the apparent lack 
of symmetry in the generalization of mirror symmetry to 
non-Calabi-Yaus bemoaned by Givental in \cite{phasesWitten}.  This
is explained in remark \ref{remark:double_dual}.

The methods of Berglund-H\"ubsch, and Batyrev-Borisov
produce mirror families.  This is to say that starting
from a family of Calabi-Yaus $(Z_t)_t$ they give
a new family $(\check{Z}_s)_s$ without 
specifying which $\check{Z}_s$ is mirror to a
given $Z_t$.  This suffices for the comparison
of Hodge numbers since these are constant in families.

The Berglund and H\"ubsch construction is easily treated
and a simple calculation gives the following theorem.

\begin{intro_theorem}
Let $(X,W)$ be obtained from the sigma model / Landau-Ginzburg model
correspondence applied to a Calabi-Yau hypersurface $Z$ in weighted
projective space. 
Then the dual Landau-Ginzburg model equals the
Landau-Ginzburg model corresponding to a member $\check{Z}$
of the Berglund-H\"ubsch mirror family.
\end{intro_theorem}

The last method we analyze is that of Batyrev and Borisov.
The starting point for their construction is a split bundle 
obtained from a nef-partition over
a Gorenstein toric Fano variety.
After some technical results concerning rational convex polyhedral
subsets, we arrive at the following theorem.

\begin{intro_theorem}
Let $(X,W)$ be obtained from the sigma model / Landau-Ginzburg model
correspondence applied to a certain  Calabi-Yau $Z_{BB}$ given by 
a global section of a split bundle defined by a nef-partition over
a  toric Fano variety.
Then the dual Landau-Ginzburg model equals the Landau-Ginzburg
model corresponding to a member of the Batyrev-Borisov
mirror family.
\end{intro_theorem}

One nice aspect about the dual Landau-Ginzburg
model is that varying the original superpotential, $W$, causes the symplectic
form $\omega' \in A_{n-1}(X')_{\Bbb{C}/\Bbb{Z}}$ to vary.  When
$X$ is obtained via the sigma model / Landau-Ginzburg correspondence,
varying the superpotential is the same as varying the complex structure
of $Z$.  This identification of complex moduli with symplectic
moduli is expected between mirror pairs.  
 
This can be used to avoid potential difficulties that arise
when the mirror candidate is singular.
For instance, every element of the Batryev-Borisov mirror
family may be singular since the ambient toric variety
may be singular.  However, varying $Z$ away from $Z_{BB}$
leads to a partial resolution of the mirror, 
thus taking some of the arbitrariness out of the choice 
of resolution.

There is a small example in section \ref{section:structure},
and we conclude the paper in with an extended example that 
makes up section \ref{examples}.  Here, the case 
of elliptic curves in $(\Bbb{P}^1)^2$ is treated from the point
of view of Givental and Batyrev-Borisov, and it is shown how
our point of view  allows for symplectic 
a resolution of these singular 
mirrors.

\smallskip
\smallskip

{\bf Acknowledgements.} This work would not have been possible without the support and 
insight of Ludmil Katzarkov.  We would also like thank Elizabeth Gasparim for her close 
reading of the text, and David Cox for pointing out several references.  
Thanks to Eric Sharpe for showing me the Berglund-H\"ubsch construction 
and explaining several ideas from physics to me.  Finally, thanks to Lev Borisov 
for pointing out a way to modify the 
construction given here.  This modification is discussed at the end of the paper.

\section{Definitions}
\label{definitions}

\subsection*{Rational convex polyhedral sets}

The theory of quasi-projective toric varieties is essentially
the same as the theory of rational convex polyhedral sets.
We will use this fact extensively throughout the paper.
Here we review  the main points we need.

Given a finite rank free abelian group $M$, we write $M_G$ 
for the tensor product $M \otimes_\Bbb{Z} G$ with an abelian group $G$.
Denote the dual, $\mathrm{Hom}_\Bbb{Z}(M, \Bbb{Z})$, by $N$.

A rational convex polyhedral set $P \subseteq M_\Bbb{R}$
is the set of solutions to a set of linear inequalities:  
\begin{equation}\{ \xi \in M_\Bbb{R} \ | 
\ \nu_j(\xi) +\alpha_j \geq 0, \ j = 1, \cdots,\
 r \}, \end{equation}
where  $\nu_j \in \{ \nu \otimes 1 \colon  M_\Bbb{R} \rightarrow
\Bbb{R} \ | \ \nu \in N \}$
and $\alpha_j \in \Bbb{R}$.

The inequalities can be packaged together into a homomorphism
$A = (\nu_1, \cdots, \nu_r) \colon M \rightarrow \Bbb{Z}^r$,
and an element $\alpha = (\alpha_1, \cdots, \alpha_r) \in \Bbb{R}^r = (\Bbb{Z}^r)_\Bbb{R}$.
With these we have $P = \{ \xi \in M_\Bbb{R} \ | \ A(\xi) + \alpha \geq 0 \}$.

A face of $P$ is either the 
intersection of $P$
with the boundary of an affine half-space containing $P$, 
or $P$ itself.  The dimension of a face is the dimension
of the real vector space given by the span of the elements
of the face.  A facet of $P$ is a face whose dimension is
one less that the dimension of $P$.

If $C$ is an arbitrary subset of $N_\Bbb{R}$, the dual
of  $C$ is the set
\begin{equation}\check{C} := \{ \xi \in M_\Bbb{R} \ |
\ \nu(\xi) \geq 0, \  \forall \ \nu \in C \}.\end{equation}

Defined similarly to the dual, is the polar of $C \subseteq N_\Bbb{R}$
\begin{equation}C^\circ := \{ \xi \in M_\Bbb{R} \ |
\ \nu(\xi) + 1 \geq 0, \ \forall \ \nu \in C \}.\end{equation}

Associated to a convex rational polyhedral set with non-empty 
interior is an {\it inward normal fan}.  This is made up 
of rational convex polyhedral cones.
A rational convex polyhedral cone is rational convex
polyhedral set closed under multiplication by $\Bbb{R}_{\geq 0}$,
and it is called strongly convex if $0$ is a vertex 
(i.e. $\{ 0 \}$ is a $0$ dimensional face).

A fan, $\Sigma$, is a non-empty finite collection of strongly convex rational polyhedral
cones $\subset N_\Bbb{R}$ such that
\begin{enumerate}
\item if $\sigma \in \Sigma$ then all faces of $\sigma$ are in $\Sigma$, and 
\item the intersection of any two cones in $\Sigma$ is also in $\Sigma$.
\end{enumerate}

The inward normal fan $\Sigma_P$ to $P$ is defined to be
the collection of cones
\begin{equation}
\sigma_f
:= \{ \nu \in N_\Bbb{R} \ | \ \nu(\xi)  =  \mathrm{min}(\nu|_{P}),
\ \forall \ \xi \in f \}
\end{equation}
for  each face $f$ of $P$.

\begin{lemma}(see for instance \cite{toricFulton})
If $P$ has non-empty interior, then $\Sigma_P$ is a fan.
\end{lemma}

We finish this discussion with some small, but useful results 
about the polar of a convex set.  

\begin{lemma}
If $C$ is convex and $0 \in \mathrm{int}(C)$, then $(C^\circ)^\circ = C$.
\end{lemma}

\begin{proof}
$C$ is determined by the affine half-spaces containing it.  These half-spaces
contain $0$ in their interior, so they have a defining inequality with constant
part $= 1$.  So we have $C = \cap_{\xi \in C^\circ} \{ \nu \in N_\Bbb{R} \ 
| \ \nu(\xi) + 1 \geq \ 0 \}$,
and the result is clear.
\end{proof}

\begin{corollary}
Assume $0 \in \mathrm{int}(P)$, then $P^\circ = \mathrm{conv}(\{ \nu_j/\alpha_j
\}_{j=0}^r \cup \{ 0 \})$.
\end{corollary}

\begin{proof}
Since $\alpha_j > 0$, $P  = 
\mathrm{conv}(\{ \nu_j/\alpha_j \}_{j=0}^r \cup \{ 0 \})^\circ$\ .
Taking polars of convex sets reverses inclusions of convex sets, and  
preserves strictness for
convex sets containing 0.
\end{proof}

\begin{corollary}
\label{redundantPolar}
Assume $0 \in \mathrm{int}(P)$.
Given $\nu_0 \in M_\Bbb{R}$
and $\alpha_0 \in \Bbb{R}_{> 0}$.
$P$ is contained in the affine half-space  
$\{ \xi \in M_\Bbb{R} \ | \ \nu_0(\xi) + \alpha_0 \geq 0 \}$ if and only if
$\nu_0/\alpha_0 \in \mathrm{conv}(\{ \nu_j/\alpha_j \}_{j=0}^r \cup \{ 0 \})$.
\end{corollary}

\smallskip

\subsection*{Toric varieties}
\label{ToricSection}

A toric variety $X$ is a normal irreducible complex algebraic variety on which
an algebraic torus $T \cong (\Bbb{C}^\times)^n$ acts and
$\exists\ x \in X$ such that $t \mapsto t \cdot x$
defines an open immersion 
\begin{equation}
\iota_x \colon  T \hookrightarrow X.
\end{equation}
Some standard references for toric varieties are 
\cite{toricOda, toricFulton, toricAudin}.

The open immersion $\iota_x$ identifies characters 
$\xi: T \rightarrow \mathbb{C}^\times$
with rational functions on $X$.  Since the divisor of a character is supported 
on the
union of $T$-invariant subvarieties, this gives the {\it character-to-divisor  map} 
\begin{equation}
\label{character_to_divisor_map}
\mathrm{div} \colon M \rightarrow \Bbb{Z}^R,
\end{equation}
where $M$ is the group of characters of $T$, and
$R := \{\rho_1, \cdots, \rho_r \}$ is the set of components  of 
$X \setminus \iota_x(T)$.
This means $\Bbb{Z}^R$ is the set of $T$-invariant divisors of $X$.

The cokernel of $\mathrm{div}$ is the Chow
group $A_{n-1}(X)$.  This group can often be identified
with the second integral cohomology group as the
following theorem and corollary state.
\begin{theorem}{(\cite[pp. 63-64]{toricFulton})}
If $Y$ is a complete toric variety,
$A_{n-1}(Y) = \mathrm{H}^2(Y; \Bbb{Z})$
and is torsion free.
\end{theorem}
\begin{corollary}
\label{Chow_equals_cohomology}
If $X$ is the total space of a split bundle of rank $c$ over a complete 
toric variety, $Y$, then $X$ is toric,
$A_{n-1}(X) = \mathrm{H}^2(X, \Bbb{Z}) =
A_{n-1-c}(Y) = \mathrm{H}^2(Y; \Bbb{Z})$,
and these groups are torsion free. 
\end{corollary}

Denote the cokernel of $\mathrm{div}$ by \begin{equation}[-] \colon 
\Bbb{Z}^R \rightarrow A_{n-1}(X),\end{equation}
and write the image of and element $d \in \Bbb{Z}^R$ by $[d] \in A_{n-1}(X).$

Consequently, in the case of corollary \ref{Chow_equals_cohomology}
we have the sequence
\begin{equation}
\label{completeCase}
0 \rightarrow M \rightarrow \Bbb{Z}^R \rightarrow \mathrm{H}^2(X; \Bbb{Z}) \rightarrow 0,
\end{equation}
which is exact.

For toric varieties, the $T$-invariant divisor
\begin{equation}
\label{canon_anticanon}
-\kappa_X := 1 \rho_1 + \cdots + 1 \rho_r\end{equation}
gives a canonical choice of representative
for the anticanonical divisor.

The group $\mathrm{Hom}_\Bbb{Z}(\Bbb{Z}^R, \Bbb{C}^\times)$
acts diagonally on $\Bbb{C}^R$, and thus so does
$\mathrm{Hom}_\Bbb{Z}(A_{n-1}(X), \Bbb{C}^\times)$
via the group homomorphism
\begin{equation}(- \ \circ [-]) \colon
\mathrm{Hom}_\Bbb{Z}(A_{n-1}(X), \Bbb{C}^\times)
\rightarrow
\mathrm{Hom}_\Bbb{Z}(\Bbb{Z}^R, \Bbb{C}^\times)\end{equation}
defined by $h \mapsto h \circ [-].$

If we restrict the induced map on the cotangent space
of the identity, 
\begin{equation}
T^*_{Id} \mathrm{Hom}_\Bbb{Z}(\Bbb{Z}^R, \Bbb{C}^\times)
\rightarrow
T^*_{Id} \mathrm{Hom}_\Bbb{Z}(A_{n-1}(X), \Bbb{C}^\times)
\end{equation}
to the lattice dual to the exponential kernel, we get
$fr \circ [-]$. 
Here $fr$ is the
projection of $A_{n-1}(X)$ onto $A_{n-1}(X)/\mathrm{torsion}$.
For the main applications considered here, $A_{n-1}(X)$ is 
torsion free, and so we can interpret $[-]$ as the 
pullback map on covectors at the identity.

Using the action of $\mathrm{Hom}_\Bbb{Z}(A_{n-1}(X), \Bbb{C}^\times)$
on $\Bbb{C}^R$,
the quotient construction for projective space generalizes to toric 
varieties.

\begin{theorem}{(Cox \cite{CoxHomog})}
A toric variety, $X$, can be obtained as a (GIT) quotient
of $\Bbb{C}^R \setminus \mathcal{P}$
by $\mathrm{Hom}_\Bbb{Z}(A_{n-1}(X), \Bbb{C}^\times)$ \cite{CoxHomog}.
$\mathcal{P}$ is a collection of coordinate subspaces determined by 
the intersection theory on $X$.
\end{theorem}

The ring of regular functions on $\Bbb{C}^R$ is called the {\it Cox
homogeneous coordinate ring} of $X$, and is generated by 
$\{X_\rho \}_{\rho \in R}$.  The degree of $X_\rho$ is $[1 \rho] \in A_{n-1}(X)$.

A more classical way to define a toric variety is from 
a fan of rational strongly convex polyhedral cones.
The construction of a toric variety goes as follows.
Fix a torus $T$ and let $M$ be its character group.   
Denote $N = \mathrm{Hom}_\Bbb{Z}(M, \Bbb{Z})$
as before.
Let $\Sigma$ be a fan of strongly convex rational 
polyhedral cones
in $N_\Bbb{R}$. Then define the 
toric variety $X(\Sigma)$ acted on by $T$, 
to be the union of affine 
charts $U_\sigma$ for each $\sigma \in \Sigma$.  Where
$U_\sigma$ is the spectrum of the subring of regular
functions $R_\sigma$ on $T$ that is generated by characters
$\xi$ in the dual cone $\check{\sigma}$.

\begin{theorem} (Sumihiro \cite{sumihiro1, sumihiro2})
Every toric variety can be obtained via the fan construction
in a unique way.
\end{theorem}

\subsection*{Functions on a toric variety}

As indicated by the character-to-divisor map 
((\ref{character_to_divisor_map}) above), 
the function theory of $X$ is understood in terms of the
function theory of the torus $T$.  The identification of
rational function on $X$ with those on $T$ depends on the
choice of $x$ defining the open immersion $\iota_x$.

If $W$ is a rational function on $X$,
\begin{equation}
\iota_x^*W =  \sum_j q_j \xi_j
\end{equation}
is the pullback to $T$, where $\xi_j \in M$.

If one makes another choice, $x'$, there is a
unique element 
$t \in T$ such that $x' = t \cdot x$. 
For the same function $W$, 
\begin{equation}
\label{pullback_x'}
\iota_{x'}^* W = \iota_{(t \cdot x)}^*W = \sum_j q_j \xi_j(t)  \xi_j .
\end{equation}

Denote by $(\mathbb{C}^\times)^{\Xi}$ the space of functions on $T$ 
with terms
$\Xi := \{ \xi_1, \cdots, \xi_s \}$. The action
$T \circlearrowright (\mathbb{C}^\times)^\Xi$ coming from the 
different choices of $x \in X$ is given by
\begin{equation}
\label{T_acts_on_function}
t \cdot (q_1 \xi_1+ \cdots+ q_s \xi_s) = 
q_1 \xi_1(t) \xi_1 + \cdots + q_s \xi_s(t) \xi_s .
\end{equation}

It is then natural to eliminate the dependence on $x$ of the expression
of $\iota_x^* W$ by thinking of $W$ as an element the quotient
\begin{equation}
\label{function_space}
W \in (\mathbb{C}^\times)^{\Xi} / T .
\end{equation}

\smallskip

\smallskip

\section{Linear data}
Associated to a toric variety is the character-to-divisor map, $\mathrm{div}$.  
It is a remarkable fact that this map is almost enough information
to recover the original toric variety.  For instance, if the toric variety 
is projective and we are also given a (very) ample divisor class, $a$, in the Chow group 
($= \mathrm{coker(div)}$), the variety can be recovered.  

We will prove this, and a more general result concerning total spaces
of split bundles over certain toric varieties.  Motivated by this result
we call the pair $(\mathrm{div}, a)$ {\it the linear data} associated to
$X$.  Here $a$ is an arbitrary element 
of $A_{n-1}(X)_\Bbb{R}$.

\begin{definition}
Precisely, we define for an abelian group  $H$, linear $H$-data.
This the following information:
\begin{enumerate}
\item a finite rank free abelian group $G$,
\item a homomorphism $C \colon G \rightarrow \Bbb{Z}^t$, and
\item an element $c \in \mathrm{coker}(C)_H$.
\end{enumerate}
\end{definition}

A key fact, that we will exploit in our construction of the dual Landau-Ginzburg
model, is that we can define linear $\Bbb{C}/\Bbb{Z}$-data 
associated to a function $W$ on a
toric variety.   Analogous to the case of projective toric varieties, 
the linear data associated to $W$  recovers $W$.

\subsection*{Linear data associated to a toric varieties}
In section \ref{definitions}, we defined a polyhedral set
from a homomorphism $A \colon M \rightarrow \Bbb{Z}^r$,
and an element $\alpha \in (\Bbb{Z}^r)_\Bbb{R}$.  If 
a polyhedral set had non-empty interior, we defined
its inward normal fan.  Finally, from a fan we defined
a toric variety.

Instead, we would like to use as our starting point the 
linear data 
\begin{equation}
(A, a).
\end{equation}

\begin{definition}
Assume there exists $\alpha \in (\Bbb{Z}^r)_\Bbb{R}$ such
that $\alpha$ maps to $a$ in $\mathrm{coker}(A)$, and
$(A,\alpha)$ defines a rational convex polyhedral with
non-empty interior.
Define $X(A,a)$ to be the toric variety defined by $A$
and $\alpha$.
\end{definition}

The lemma below shows that $X(A,a)$ is independent of the choice of $\alpha$.
Specifically, different choices of $\alpha$ correspond to translation of the 
rational convex polyhedral set 
by an element of $M_\Bbb{R}$, and thus give the same inward normal fan.

The following notation for our polyhedral sets will be used throughout 
the paper.
\begin{equation}
P_\alpha = \{ \xi \in M_\Bbb{R} \ | \ A(\xi) + \alpha \geq 0 \} .
\end{equation}

\begin{lemma}
If $\alpha - \alpha' = A(\xi_0)$ then $P_{\alpha'} = \xi_0 + P_\alpha$.
\end{lemma}

\begin{proof}
$\xi \in P_\alpha \iff A(\xi) + \alpha \geq 0 \iff A(\xi) + A(\xi_0) + \alpha' \geq 0
\iff A(\xi + \xi_0) + \alpha' \geq 0 \iff \xi + \xi_0 \in P_{\alpha'}$ .
\end{proof}

\begin{theorem}
\label{projPropF}
If $Y$ is projective and $a \in A_{n-1}(Y)$ corresponds to a very ample line bundle
then $Y = X(\mathrm{div}, a)$.
\end{theorem}

\begin{proof}
$P_\alpha$ is the polytope of $T$-linearized global sections of the
very ample line bundle corresponding to $a$.  The result is then the standard
fact that the fan of $Y$ is the normal cone fan of this polytope.
\end{proof}

The following lemma helps us to get a concrete handle on the linear 
data associated with a line bundle over a toric variety.
Note that we identify vectors with one column 
matrices.  
Therefore homomorphisms are written as matrices that multiply
vectors from the left.

\begin{lemma}
\label{A-formula}
If $D$ is a $T$-invariant Cartier divisor and $E$ is the total space of a line bundle
$\mathcal{O}_Y(-D)$ over a toric variety $Y$, then
the character group of $E$ is \begin{equation}M_E = M_Y \oplus \Bbb{Z} \xi, \end{equation}  
where
$\xi$ is a rational section of $\mathcal{O}_Y(D)$ whose
divisor is $D$.  The $T$-invariant Weil divisors of $E$
are the preimages under $p$ of the $T$-invariant Weil divisors
of $Y$ as well as the image of the zero section $\sigma_0 \colon Y \rightarrow E$,
and so
\begin{equation}
\label{preAformula}
\mathrm{div}_E =
\left[
\begin{array}{ccc}
\mathrm{div}_Y & | & D \\
0 & | & \sigma_0(Y)
\end{array}
\right]
\end{equation}
with respect to the decomposition of $M_E$ above and
$\Bbb{Z}^{R_E} = \Bbb{Z}^{R_Y} \oplus \Bbb{Z} \sigma_0(Y)$.
\end{lemma}

\begin{proof}
The function $\xi$ is a section of a line bundle, so it
vanishes both over $D$ when $\xi \equiv 0$ and along
the zero section $\sigma_0(Y)$.  Other characters
are pullbacks, so they vanish as they did on $Y$.
\end{proof}

\begin{theorem}
\label{propertyF}
Assume  $Y = X(\mathrm{div}_Y, a)$, and
$E = \mathrm{Tot} \ \Gamma(Y, \mathcal{O}_Y(-D_1) \oplus \cdots \oplus \mathcal{O}_Y(-D_c))$,
and the 
$D_j$ are $T$-invariant Cartier divisors with $|D_j|$ basepoint free,
then $E = X(\mathrm{div}_E, p^* a)$.  Where $p \colon E \rightarrow Y$ is the projection.
\end{theorem}

\begin{proof}
The pullback of a base-point free linear system
is base-point free, so by induction, it suffices to check
this for when $c=1$.

In this case $E = \mathrm{Tot}(\mathcal{O}_Y(-D))$.
The pullback map on divisors sending $\rho \mapsto p^{-1}(\rho)$
induces the pullback map $p^* \colon A_{n-1}(Y) \rightarrow A_{n}(E)$,
and so if $a = [\alpha]$ then $p^*a = [(\alpha,0)]$.

The polyhedral set $P_{(\alpha, 0)}$ is defined by the inequality
$\mathrm{div}_E(\xi \oplus s \xi) + (\alpha, 0)= (\mathrm{div}_Y(m)+ s D +\alpha) 
\oplus s  \sigma_0(Y) \geq 0$.
Immediate observations we make are  $s \geq 0$ and $P_{(\alpha,0)} \cap \{ s = s_0 \}
= P_{\alpha + s_0 D} \oplus s_0$. The second observation becomes important
in light of the fact that $|D|$ base-point
free implies that $|s_0 D|$ is base-point free for $s_0 \in \Bbb{Z}_{>0}$,
and very-ample plus base-point free implies very-ample.  So $P_{\alpha + s_0 D}$
has the same face structure as $P_\alpha$ and consequently the polytope
$P_{p^*a}$ has the face structure of $P_\alpha \times [0,\infty)$.

The 1-cones in the fan of $E$ and $X(\mathrm{div}_E,p^*a)$
are the same, and the agreement of face structures implies $X(\mathrm{div}_E,p^*a)$
is covered by charts centered at torus fixed points of $Y$ just like $E$.
This implies the $n$-cones agree as well, and thus the full fans.
\end{proof}

\subsection*{Linear data associated to rational functions on toric varieties}
In section \ref{definitions} we introduced 
the homomorphism $T \rightarrow (\Bbb{C}^\times)^\Xi$,
given by $t \mapsto \xi_1(t) \xi_1 + \cdots + \xi_s(t) \xi_s$.
This was used to write the space of rational functions on a toric 
variety with terms $\Xi = \{ \xi_1, \cdots, \xi_s \}$
as a quotient $(\Bbb{C}^\times)^\Xi / T$.
 
The homomorphism induces a map between Lie algebras
that is integral with respect to the 
kernels of the exponential maps.  
The kernel of the exponential map on $T$
is typically denoted $N$, and is naturally identified
with the dual space to the characters, $\mathrm{Hom}_\Bbb{Z}(M, \Bbb{Z})$,
and the one parameter subgroups $\Bbb{C} \rightarrow T$ (see for instance
\cite{toricFulton}).
If we write $\Bbb{Z}^\Xi$ for the kernel of the exponental map 
on $(\Bbb{C}^\times)^\Xi$, the homomorphism between these lattices is 
\begin{equation}
\mathrm{mon} \colon N \rightarrow \mathbb{Z}^{\Xi},
\end{equation}
where $\mathrm{mon}$ is short for ``infinitesimal action on monomials''.

The isomorphism $\exp( 2 \pi i - ) \colon \Bbb{C}/\Bbb{Z} \rightarrow \Bbb{C}^\times$ 
induces an isomorphism between the exact sequences
\begin{equation}
N_{\Bbb{C}/\Bbb{Z}} 
\stackrel{\mathrm{mon}}{\rightarrow} 
(\Bbb{Z}^\Xi)_{\Bbb{C}/\Bbb{Z}} \rightarrow 
\mathrm{coker(mon)}_{\Bbb{C}/\Bbb{Z}} 
\rightarrow 0
\end{equation}
and 
\begin{equation}
T
\rightarrow
(\Bbb{C}^\times)^\Xi
\rightarrow 
(\Bbb{C}^\times)^\Xi /T
\rightarrow 1 .
\end{equation}
This uses the standard identification of $T$ with $N_{\Bbb{C}^\times}$.

\begin{definition}
Let $W$ be a rational function on a toric variety, $X$, with terms $\Xi$.
Denote by  $L \in \mathrm{coker(mon)}_{\Bbb{C}/\Bbb{Z}}$ 
the point corresponding to $W$.
 We define the linear $\Bbb{C}/\Bbb{Z}$-data of $W$
to be the pair $(\mathrm{mon}, L)$.
\end{definition}

\begin{lemma}
Let $\lambda \in \Bbb{C}^\Xi$ such that $\lambda$ maps to $L$ in 
$\mathrm{coker(mon)}_{\Bbb{C}/\Bbb{Z}}$. Let 
$W(\mathrm{mon}, \lambda)$ be the function on $T$ determined by $\mathrm{mon}$ and $\lambda$,
then there is a unique $x \in X$ such that $\iota^* W = W(\mathrm{mon},\lambda)$.
\end{lemma}

\begin{proof}
This is an immediate consequence of the definition of $\mathrm{mon}$,
or equivalently the definition of the homomorphism 
$T \rightarrow (\Bbb{C}^\times)^\Xi$.
\end{proof}

It is worth pointing out that $W(\mathrm{mon},\lambda) 
= \sum_j \exp(2 \pi i \lambda_j) \xi_j$.

\begin{remark}
Despite the fact that the linear data of a rational function 
only determines the function up to the action of $T$, we
will refer to the ``rational function'' $W(\mathrm{mon}, L)$.
This is because the action of $T$ will be viewed as a coordinate
choice and not intrinsic to the situation.
\end{remark}

\section{Toric Landau-Ginzburg models}  
\label{section:toricLGmodels}

In the introduction,
a  Landau-Ginzburg model was a superconformal theory defined by K\"ahler manifold, 
$X$, equipped with
a holomorphic function $W \colon X \rightarrow \Bbb{C}$, where  $W$ is called
the superpotential.  It is common to add an additional piece of information
called the {\it B-field}.  This is simply a cohomology class 
$b \in \mathrm{H}^2(X; \Bbb{R}/\Bbb{Z})$.

We will use the term {\it toric Landau-Ginzburg model} to mean the following
information: 
\begin{enumerate}
\item a toric variety $X$,
\item an element $K \in A_{n-1}(X)_{\Bbb{C}/\Bbb{Z}}$, and
\item a regular function $W \colon X \rightarrow \Bbb{C}$.  
\end{enumerate}
The class $K = b + ia$ is thought of as the B-field, $b$, and
the K\"ahler class, $a$, packaged together into a single ``complexified
K\"ahler class''.

In the cases where corollary \ref{Chow_equals_cohomology} applies, 
this is the same as a Landau-Ginzburg model except that we have retained 
only the class of the K\"ahler form and forgotten the form itself.

\subsection*{Toric Landau-Ginzburg models defined by linear data}

Associated to a toric Landau-Ginzburg model is 
its {\it linear data}:
\begin{enumerate}
\item the linear $\Bbb{C}/\Bbb{Z}$-data $(\mathrm{div}, K)$, and 
\item the linear $\Bbb{C}/\Bbb{Z}$-data $(\mathrm{mon}, L)$,
\end{enumerate}

One can also start from two sets of linear $\Bbb{C}/\Bbb{Z}$-data,
$(A,K), (B,L)$, and, provided
the polyhedral set defined by  $(A, \Im(K))$ has non-empty interior,
define a toric variety $X(A,\Im(K))$ and rational function $W(B,L)$.
Here $\Im(K)$ denotes the imaginary part of $K$.
This is not quite a toric 
Landau-Ginzburg model, because $K$ may not define a complexified
K\"ahler class, or $W$ may not be regular.

\begin{definition}
To avoid problematic situations, we will define linear $\Bbb{R}-$data $(C,c)$
to be {\it kopasetic} if 
\begin{enumerate}
\item the polyhedral set defined by $(C, c)$ has non-empty interior, and 
\item there exists a surjection $k: \Bbb{Z}^r \rightarrow \Bbb{Z}^{R_{X(C, c)}}$
that sends standard generators to standard generators or zero, and the diagram
\begin{equation}
\label{Gmap}
\begin{diagram}
\node{M} \arrow{e,t}{C} \arrow{se,b}{\mathrm{div}_{X(C,c)}}
\node{\Bbb{Z}^r} \arrow{s,r}{k}\\
\node[2]{\quad \quad \Bbb{Z}^{R_{X(C,c)}}}
\end{diagram}
\end{equation}
commutes.
\end{enumerate}
\end{definition}
We will also denote by $k$ the induced map on the
cokernels.

Note that $k$ is essentially unique in the sense that it
is constructed, by eliminating unnecessary inequalities
from the family $A(\xi) + \gamma \geq 0$, where $\gamma$
is a lift of $c$.  Non-uniqueness arises if two of the
inequalities are identical, in which case we drop one
of them.

To address the regularity of $W(B,L)$ we first make the 
simple observation that a rational function $W$ is 
regular if and only if all its monomials are regular,
and a monomial $\xi \in M_X$ is regular if and only if
\begin{equation}
\label{regularCharacter}
\mathrm{div}(\xi) \geq 0 .
\end{equation}
One can easily use these facts to check the statement: 

\begin{lemma}
$W$ is regular
$\iff \mathrm{div} \circ \mathrm{mon}^\tau \geq 0$.
\end{lemma}

\begin{definition}
With this in mind we define a pair of linear $\Bbb{C}/\Bbb{Z}$-data,
$(A,K)$, and $(B,L)$, to be {\it kopasetic} if 
\begin{enumerate}
\item $(A,\Im(K))$ is kopasetic, and
\item $A \circ B^\tau \geq 0$.
\end{enumerate}
\end{definition}

\begin{definition}
Given a pair of kopasetic linear $\Bbb{C}/\Bbb{Z}$-data,
$(A,K)$, and $(B,L)$, we define a toric Landau-Ginzburg model
with
\begin{enumerate}
\item the toric variety $X(A, \Im(K))$,
\item the regular function $W(B,L)$, and
\item the complexified K\"ahler class $k(K)$.
\end{enumerate}
\end{definition}

\subsection*{The dual toric Landau-Ginzburg model}
So far we have seen how one can extract linear data from
a toric Landau-Ginzburg data, and under kopasetic conditions 
define a toric Landau-Ginzburg model from a pair of linear 
$\Bbb{C}/\Bbb{Z}$-data.  

On the level of linear data, there 
is a simple involution defined by simply interchanging
the roles of $(\mathrm{div}, K), (\mathrm{mon}, L)$.

\begin{definition}
Let $(A,K), (B,L)$ be a pair of linear $\Bbb{C}/\Bbb{Z}$-data.
Define the  {\it dual} pair $(A', K'), (B',L')$, by
\begin{equation}
(A', K') = (B,L), \ \mathrm{and}, \ (B',L') = (A, K).
\end{equation}
\end{definition}

\begin{definition}
Given a toric Landau-Ginzburg model, $(X, W, K)$, assume that the dual to its linear data
is kopasetic.  Define the {\it dual toric Landau-Ginzburg model} to be the
toric Landau-Ginzburg model defined by the dual linear data.  Denote the
dual by $(X', W', K')$.
\end{definition}

\begin{remark}
\label{kopasetic_test_remark}
To check that the dual data is kopasetic, one only needs to 
verify that $(A', \Im(K'))$ is kopasetic since 
$A' \circ B'^\tau  = (A \circ B^\tau)^\tau$,
and  $W$ is regular so $(A \circ B^\tau) \geq 0$.
\end{remark}

\smallskip

\section{The sigma model / Landau-Ginzburg model correspondence}
\label{correspondence}
\label{AandB}

In physics, a {\it sigma model} is a superconformal theory
defined by K\"ahler manifold $Z$ equipped
with a B-field $b \in \mathrm{H}^2(Z, \Bbb{R}/\Bbb{Z})$.
We will take a sigma model to be an algebraic variety with a choice 
of complexified K\"ahler class in its Chow group of divisors.  
This class will be a pullback from such a class on an ambient space.

One can produce interesting toric Landau-Ginzburg models 
from the sigma model / Landau-Ginzburg model correspondence applied to 
a complete intersection in a toric variety.  

As mentioned in the introduction, when a subvariety $Z$ of a variety $Y$
is the zero locus of a global section $\mathrm{w}$ of a vector bundle $\mathcal{V}$, 
we can define a Landau-Ginzburg model. Strictly speaking, this information
defines a morphism $W \colon X \rightarrow \Bbb{C}$ from the total space 
of the dual bundle $\mathcal{V}^\vee$, $X = \mathrm{Tot}(\mathcal{V}^\vee)$,
to $\Bbb{C}$.
The morphism is defined by the pairing 
$\mathcal{V} \otimes \mathcal{V}^\vee \rightarrow \mathcal{O}_X$ that identifies
sections of $\mathcal{V}$ with functions on $X$.

If $Y$ is equipped with complexifed divisor class $K \in A_{n-1}(Y)_{\Bbb{C}/\Bbb{Z}}$, 
then we can pull it back to define a complexified divisor class on $X$.
Also if $Y$ is toric and $\mathcal{V}$ is split, then it is easy to see that
$X$ is a toric variety.  So we make the following definition.

\begin{definition}
Let $\mathcal{V}$ be a rank $c$ split bundle over an $n$-dimensional toric variety $Y$.  
Let $\mathrm{w}$ be a global section of $\mathcal{V}$, and let 
$K \in A_{n-1}(Y)_{\Bbb{C}/\Bbb{Z}}$.  Finally let $Z$ be the zero locus of  
$\mathrm{w}$.  We define {\it toric Landau-Ginzburg model corresponding to $(Z,K)$} 
to be $(X, W, K)$, where $X = \mathrm{Tot}(\mathcal{V}^\vee)$, $W$ is the function defined 
by $\mathrm{w}$, and $K$ is the pullback to $A_{n+c-1}(X)$ of $K \in A_{n-1}(Y)$.
\end{definition}

Note that the resulting toric Landau-Ginzburg model actually depends on $Y$, 
$\mathrm{w} \in \Gamma(Y,\mathcal{V})$, and $K$,
rather than $(Z, K)$.
Also notice $A_{n+c-1}(X) =  A_{n-1}(Y)$.

\begin{remark}
As explained to us \cite{Sharpe:private}:
in the literature, when studying a sigma model $Z$
it is a common trick to move to a ``Landau-Ginzburg point''
of the gauged linear sigma model K\"ahler 
moduli space of $Z$ to obtain a Landau-Ginzburg model
whose $B$-twist is the same as that of $Z$.  
We are not using this Landau-Ginzburg model.  The Landau-Ginzburg
model $(X,W,K)$ and the sigma model $(Z,K)$ lie in the same 
universality class
and so truly define the same superconformal theory.
See \cite{SharpeGuffin:A-twistLG} for a detailed treatment. 
\end{remark}

\subsection*{Linear data associated to $(X,W,K)$}

An immediate consequence of lemma \ref{A-formula} is the following formula
for the map $\mathrm{div}_X$.

\begin{corollary}
\label{bigA-formula}
(of lemma \ref{A-formula})
If $D_1, \cdots, D_c$ are $T$-invariant Cartier divisors and $X$ is the total
space of the split bundle
$\mathcal{O}_Y(-D_1) \oplus \cdots \oplus \mathcal{O}_Y(-D_c)$ over a toric variety $Y$, then
the character group of $X$ is \begin{equation}M_X = M_Y \oplus \Bbb{Z} \sigma_1 \oplus \cdots
\oplus \Bbb{Z} \sigma_c, \end{equation}  where
$\sigma_j$ is a rational section of $\mathcal{O}_Y(D_j)$ whose
divisor is $D_j$.  The $T$-invariant Weil divisors of $X$
are the preimages under $p$ of the $T$-invariant Weil divisors
of $Y$ as well as the total spaces $X_j$ of the $c$ subbundles $\mathcal{V}^\vee_j$, where 
$\mathcal{V}^\vee_j$ 
is the dual bundle to
$\mathrm{ker}(\pi_j \colon  \mathcal{V} \rightarrow \mathcal{O}(D_j))$.  Furthermore,

\begin{equation}
\label{Lequation}
\mathrm{div}_X =
\left[
\begin{array}{ccccccc}
\mathrm{div}_Y & | & D_1 & | & \cdots & | & D_c\\
0 & | &  &  & \mathrm{Id} &  &
\end{array}
\right]
\end{equation}
with respect to the decomposition of $M_X$ above and
$\Bbb{Z}^{R_X} = \Bbb{Z}^{R_Y} \oplus \Bbb{Z} X_1 \oplus \cdots
\oplus \Bbb{Z} X_c$.
\end{corollary}

\begin{proof}
This formula is obtained by repeated application of lemma \ref{A-formula}.
\end{proof}

The linear data corresponding to the superpotential $W$ is easily obtained via the
following lemma and the standard practice of identification of $T$-linearized global
sections of $\mathcal{O}_Y(D)$ with integral points of the polytope $P_D$.

Rather than $\mathrm{mon} \colon N \rightarrow \Bbb{Z}^\Xi$,
it is easier to write down the transpose, $\mathrm{mon}^\tau$.
Using the usual identification of characters and covectors, 
\begin{equation}
\xi \leftrightarrow d\xi|_{Id} ,
\end{equation}
and of the transpose and the pullback, we get the lemma below.

\begin{lemma}
Let $X$ be a toric variety with character group $M$, and 
$W$ a rational function on $X$ with terms $\Xi$.
$\mathrm{mon}^\tau \colon  
\mathrm{Hom}_\mathbb{Z}(\mathbb{Z}^\Xi, \mathbb{Z}) \rightarrow M$, 
takes the basis element dual to $1\xi_j \in \mathbb{Z}^{\Xi}$  
to $\xi_j \in M$.
\end{lemma}

\begin{definition}
The set of terms of $W$, $\Xi_W$, is naturally identified with a subset of 
$\coprod_j (P_{D_j} \cap M_Y)$. Where $P_{D_j}$ is the polytope corresponding
to $D_j \in \Bbb{Z}^{R_Y}$.
If $\Xi_W$ is in bijection with this set, we say $\mathrm{w}$ is {\it generic}.  
\end{definition}

\begin{lemma}
\label{B-formula}
In the generic case the transpose map,
$\mathrm{mon}^\tau$, is given by the matrix
\begin{equation}
\left[
\begin{array}{ccccc}
P_{D_1} \cap M_Y & | & \cdots & | &  P_{D_c} \cap M_Y \\
\xi_1  & | & \cdots & | & \xi_c
\end{array}
\right]
\end{equation}
with respect to the decomposition of $M_X$ above and
the identification $\mathrm{Hom}_\Bbb{Z}(\Bbb{Z}^{\Xi_W},\Bbb{Z}) =
\mathrm{Hom}_\Bbb{Z}(\Bbb{Z}^{(P_{D_1} \cap M)}, \Bbb{Z}) \oplus \cdots
\oplus \mathrm{Hom}_\Bbb{Z}(\Bbb{Z}^{(P_{D_c} \cap M)}, \Bbb{Z})$.
\end{lemma}

\begin{proof}
A one parameter subgroup acts on a monomial coefficient by multiplication
of the coefficient by the subgroup plugged into the monomial itself.
Therefore the transpose simply picks out the appropriate monomial.
\end{proof}

\begin{corollary}
\label{B-block}
If we write $P^\times_{D_j}$ for $P_{D_j} \setminus  \{ 0 \}$ and $0_j$ for the $j^{th}$ zero in
$\coprod_j (P_{D_j} \cap M_Y)$, the matrix $\mathrm{mon}^\tau$ takes the form
\begin{equation}
\left[
\begin{array}{ccccccc}
P^\times_{D_1} \cap M_Y & | & \cdots & | &  P^\times_{D_c} \cap M_Y & | & 0 \\
\sigma_1  & | & \cdots & | & \sigma_c  & | &  \mathrm{Id}
\end{array}
\right],
\end{equation}
where now $\mathrm{Hom}_\Bbb{Z}(\Bbb{Z}^{\Xi_W},\Bbb{Z}) =$
$$\mathrm{Hom}_\Bbb{Z}(\Bbb{Z}^{(P^\times_{D_1} \cap M)},\Bbb{Z}) \oplus \cdots
\oplus \mathrm{Hom}(\Bbb{Z}^{(P^\times_{D_c} \cap M)}, \Bbb{Z}) \oplus
\mathrm{Hom}_\Bbb{Z}(\Bbb{Z} 0_1, \Bbb{Z}) \oplus \cdots \oplus \mathrm{Hom}_\Bbb{Z}(\Bbb{Z} 0_c, \Bbb{Z}) .$$
\end{corollary}

\begin{proof}
This is just the matrix of lemma \ref{B-formula} with the columns permuted.
\end{proof}

\begin{definition}
Following the decompositions above,  we write $\Xi^\times := \coprod_j(P^\times_{D_j} \cap M_Y)$, and so
 $\Bbb{Z}^{\Xi_W}  = \Bbb{Z}^{\Xi^\times}
\oplus \Bbb{Z} 0_1 \oplus \cdots \oplus \Bbb{Z} 0_c$.
\end{definition}

\begin{definition}
$N_X$ also has a natural decomposition in terms of the one parameter subgroups associated to
$Y$ and the summands of $\mathcal{V}^\vee$ : \begin{equation}
\label{N-decomp}
N_X = N_Y \oplus \Bbb{Z} \phi_1 \oplus \cdots
\oplus \Bbb{Z} \phi_c. \end{equation} Here
$\phi_j$ is just multiplication by $\Bbb{C}^\times$ along the $j^{\mathrm{th}}$
summand of $\mathcal{V}^\vee$ and the elements of $N_Y$
are constant in the fiber direction.  
\end{definition}

\begin{remark}
In this decomposition
$\Bbb{Z} \sigma_1 \oplus \cdots \oplus \Bbb{Z} \sigma_c \subset N_Y^\perp$, and
thus the pairing between $M_X$ and $N_X$ is the sum of the pairing for $Y$ and pairing
between $\Bbb{Z} \sigma_1 \oplus \cdots \oplus \Bbb{Z} \sigma_c$ and
$\Bbb{Z} \phi_1 \oplus \cdots \oplus \Bbb{Z} \phi_c$ where $\langle \sigma_i, \phi_j  \rangle =
\delta(i-j)$.
\end{remark}

\section{Toric Landau-Ginzburg models dual to sigma models: existance}

Corollaries \ref{bigA-formula} and \ref{B-block}, give a very concrete picture
of the linear data of a toric Landau-Ginzburg model 
$(X,W,K)$ obtained from a complete 
intersection $Z$ determined by a global section $\mathrm{w}$ of a bundle $\mathcal{V}$ over 
a toric variety $Y$.  

In this  section, we show the dual data, $(A',K'),$ and $(B',L')$
is kopasetic provided $(A', \Im(K'))$ 
is kopasetic, as mentioned in remark \ref{kopasetic_test_remark}.  This involves checking two things
\begin{enumerate}
\item the polytope defined by $(A', \Im(K'))$ has non-empty interior, and
\item there exists a surjection $k$ as in equation (\ref{Gmap}).
\end{enumerate}

\begin{lemma}
Assuming the codimension, $c$, of $Z$  is positive, $(A', \Im(K'))$ is 
kopasetic for all values of $\Im(K')$.
\end{lemma}

\begin{proof}
The existence of $k$ is immediate, since the columns of $\mathrm{mon}_W^\tau$
distinct an primitive.  Now let $\alpha'$ denote a lift of $\Im(K')$,
and consider the polyhedral set $P_{\alpha'}$.  The inward normals of 
facets of $P_{\alpha'}$ form a subset of the columns of $\mathrm{mon}_W^\tau$.  
These normals all live in the interior of the half-space of 
$N'_\Bbb{R} = (M_X)_\Bbb{R}$, 
$H = \{ \xi \in (M_X)_\Bbb{R} \ | \ \sum_j \phi_j(\xi) \geq 0\}$,
where the $\phi_j$ are as in equation (\ref{N-decomp}).  The lemma
now follows from the following statement: {\it The intersection
of affine half-spaces, such that the non-negative $\Bbb{R}$-span
of their normals does not contain a half-space, is non-empty.}

Define $C$ to be the cone dual to the cone defined by the 
non-negative span of the columns of  $\mathrm{mon}^\tau_W$.
$C$ has $0$ as an apex since its dual is contained
in a half-space, and more importantly it has a non-empty 
interior since its properly contained in a half-space.
If $p \in P_\alpha$ is is clear that 
$p + C \subset P_\alpha$ since the affine half-spaces
defining $P_\alpha$ can only be translates away from $p$.
On the other hand $P_\alpha$ can be obtained by intersection translates
of  half-planes containing $C$, so it is clear that 
$P_\alpha \neq \varnothing$.
\end{proof}

\begin{corollary}
A toric Landau-Ginzburg model 
$(X,W,K)$ obtained from a complete 
intersection $Z \neq Y$ determined by a 
global section $\mathrm{w}$ of a bundle $\mathcal{V}$ over 
a toric variety $Y$ has a dual 
toric Landau-Ginzburg model $(X',W',K')$.
\end{corollary}

\section{Toric Landau-Ginzburg models dual to sigma models: structure}
\label{section:structure}

The toric variety $X'$ is very closely 
related to the total space $E'$ of 
some vector bundle $(\mathcal{V}')^\vee$
over a toric variety $Y'$.  We will
define these objects below.  

We proceed to give sufficient conditions 
under which $X' = E'$, first
in terms of the columns of $\mathrm{mon}_W^\tau$,
or equivalently the rows of $A'$.
Finally, if $X' = E'$ then we give conditions
under which $W'$ comes from a global section
of $\mathcal{V}'$.  In other words, we 
give conditions under which we can construct
a dual sigma model $(Z', K')$ to the sigma
model $(Z,K)$ corresponding to $(X,W,K)$.

\subsection*{The toric varieties $Y'$ and $E'$.}

Corollary \ref{B-block}, describing $\mathrm{mon}^\tau_W$, 
implies that since 
$
A' = \mathrm{mon}_W
$ it has the  
block form
\begin{equation}
A' =
\left[
\begin{array}{ccc}
d' & | & D' \\
0  & | & \mathrm{Id}
\end{array}
\right].
\end{equation}
This is with respect to the decompositions
$N_X = N_Y \oplus \Bbb{Z} \phi_1 \oplus 
\cdots \oplus \Bbb{Z} \phi_c$
and 
$\Bbb{Z}^\Xi = \Bbb{Z}^{\Xi^\times} \oplus 
\Bbb{Z} 0_1 \oplus \cdots \oplus \Bbb{Z} 0_c$.

The identity matrix in the lower right block guarantees 
$\mathrm{coker}(A') = \mathrm{coker}(d')$.  Thus
we can find a lift of $\Im(K')$ of the form
$(\alpha', 0)$, where $\alpha' \in \Bbb{Z}^{\Xi^\times}$.
We would like to make the definition $Y' = X(d', \Im(K'))$,
but unfortunately, there is no guarantee that $(d', \Im(K'))$
is kopasetic.

We will make the assumption that $(d', \Im(K'))$
is kopasetic.  This is a comfortable assumption to make
in light of the following lemma.  

\begin{lemma}
There is a non-empty open cone $U_{kopa}$ of $\mathrm{coker}(d')$
such that $(d',a')$ is kopasetic if and only if  $a' \in U_{kopa}$.
\end{lemma}

\begin{proof}
An affine half-space can be translated to contain the origin within
its interior. Choose such an arrangement for all the affine half-spaces
with inward normals coming from the row of $d'$. This corresponds
to a point $a'_0 \in \mathrm{coker}(d')$. The origin remains in the interior 
of the intersection
the of the half-spaces for small deformations of $a'$ such that $(d',a')$
is kopasetic.
\end{proof}

Note that it is standard to interpret moving from one value $a'_0$ to another $a'$ 
as a deformation of the symplectic structure of the
toric variety  $X(d',a'_0)$. This may involve birational transformations
corresponding to when the volumes of curves (or higher dimensional subvarieties)
shrink to $0$.

\begin{definition}
Define the 
toric variety 
\begin{equation}
Y' = X(d',\Im(K')). 
\end{equation}
Denote 
the map of equation (\ref{Gmap}) by $k_{Y'}$ and 
the columns of $D'$ by $D'_j$.  These are
elements of $\Bbb{Z}^{\Xi^\times}$.
We now have the vector bundle
\begin{equation}\mathcal{V}' = \mathcal{O}_{Y'}(k_{Y'}(D'_1)) \oplus
\cdots \oplus \mathcal{O}_{Y'}(k_{Y'}(D'_c)), \end{equation}
the toric variety
\begin{equation}E' =
\mathrm{Tot}(\mathrm{Hom}_{\mathcal{O}_{Y'}}(\mathcal{V}', \mathcal{O}_{Y'})),
\end{equation}
and a function $W'$ defined on it by $B'$ and $L'$.
\end{definition}

\subsection*{Conditions for $X'=E'$}

\begin{definition}
We will spend the rest of the section discussing the ``rows''
of the homomorphisms $A' \colon  M_{X'} \rightarrow \Bbb{Z}^\Xi$ 
and 
$d' \colon M_{Y'} \rightarrow \Bbb{Z}^{\Xi^\times}$.  By this we 
mean the images 
under the transpose of the standard basis vectors.  The rows
of $d'$ may be a  multiset.
\end{definition}

Keep in mind that we will be using the follow identifications.
\begin{equation}
\begin{array}{l}
M_Y = N_{Y'} \ , \\
N_Y = M_{Y'} \ , \\
M_X = N_{E'} \ , \ \mathrm{and} \\
N_X = M_{E'} \ .
\end{array}
\end{equation}

The strategy for comparing
$X'$ and $E'$ is based on comparing $\mathrm{div}_{X'}$ and $\mathrm{div}_{E'}$.  
Both of these toric varieties are
defined from  rational convex polyhedral sets.
The defining inequalities for $X'$ come directly
from the matrix $A' = \mathrm{mon}_W$, 
and the element $a' \in \mathrm{coker}(A')_{\Bbb{C}/\Bbb{Z}}$.

On the other hand $E'$ is formed by first using
the upper left block, $d'$, of $A'$ to define the toric
variety $Y'$. Then a certain submatrix of the upper right
block $D'$, of the $A'$ are treated as divisors and we
get the family of inequalities coming from the 
formula in equation (\ref{bigA-formula}).

We will compare 
$$\mathrm{div}_{X'} = k_{X'} \circ A'$$   and 
$$
\mathrm{div}_{E'} = 
\left[
\begin{array}{ccc}
k_{Y'} \circ d' & | & k_{Y'} \circ D' \\
0 &|& \mathrm{Id} 
\end{array}
\right].
$$
Post-composition with $k$ results in deleting rows whose affine 
half-space is not necessary for defining the polytope.
We need some conditions under which
the rows deleted from $A'$ using $k_{X'}$ are the same of those
deleted using $k_{Y'}$. Recall, the rows of $A'$ are of the form 
$(P_{D_j} \cap M_Y) \times \{ \sigma_j \}$.

The only case when such a comparison make sense is when $a'$ lifts to 
$(\alpha', 0) \in 
\Bbb{Z}^{\Xi^\times} \oplus \Bbb{Z} 0_1 \oplus \cdots \oplus \Bbb{Z} 0_c$, 
such that $\alpha' > 0$, as pointed out in the lemma below.

\begin{lemma}
If $E'$ is defined (i.e. $(d',a')$ is kopasetic), then $a'$ has a lift as above.
\end{lemma}

\begin{proof}
Choose $\alpha'$ so that $0$ lies in the interior of the
polyhedral set defining $Y'$.
\end{proof}

\begin{definition}
Let $\alpha' = (\alpha'_1, \cdots, \alpha'_c)$ with respect to the
decomposition $\Xi^\times = \coprod_j (\Gamma(Y, \mathcal{O}_Y(D_j)) \cap \Xi^\times)$.
Furthermore, we can write $\alpha'_j$ in components as 
$\alpha'_j = (\alpha'_{(\nu, \sigma_j)})_{\nu \in P^\times_{D_j}}$.
\end{definition}

\begin{lemma}
The facets of the polytope of $X'$ correspond to the 
non-zero vertices of the convex hull $C_j$ points 
\begin{equation}
\{ (\nu, \sigma_j) / \alpha'_{(\nu,\sigma_j)} \ | \ \nu \in P^\times_{D_j} \cap M_Y \} \cup 
\{ \lambda (0,\sigma_j) \ | \ 0 \leq \lambda \in \Bbb{R} \}, 
\end{equation}
including the vertex 
``$\infty (0, \sigma_j)$''.
\end{lemma}

\begin{proof}
If we deform $(\alpha', 0)$ to $(\alpha', \epsilon)$, we can use corollary \ref{redundantPolar}
to write the dual polytope 
to the one defined
by $A'$ can be expressed as the convex hull of the points 
$(\nu,\sigma_j)/ \alpha'_{(\nu,\sigma_j)}$ and $(0,\sigma_j)/\epsilon$, 
$j = 1, \cdots, c$.

If a convex polyhedral set $C$ in a vectorspace 
$V \oplus W$ defined by points that lie in either
$V \oplus \{ 0 \}$ 
or 
$\{ 0 \} \oplus W$. 
Has vertices that are exactly
the vertices of convex hull of the points in $V \oplus \{ 0 \}$ and 
the vertices of the convex hull of the points 
in $ C \cap \{ 0 \} \oplus W$.

Putting these facts together and taking the limit $\epsilon \rightarrow 0$
gives the result.
\end{proof}

\begin{definition}
\label{definition:V_j}
Denote the projection 
$(M_{Y'})_\Bbb{R} \oplus \Bbb{R}\sigma_j \rightarrow (M_{Y'})_\Bbb{R}$,
by $\pi_j$, the vertices of $C_j$ by $V_j$, and the set 
$V_j \setminus \{ (0,0), (0,\sigma_j)\}$  by $V_j^\times$.
\end{definition}

\begin{theorem}
A row $\nu \in P^\times_{D_j}$ appears in $\mathrm{div}_{Y'}$
iff $\pi_j((\nu, \sigma_j)/\alpha_\nu)$ defines a non-zero vertex of 
$\mathrm{conv}(\{ 0 \} \cup \bigcup_j \pi_j(V_j))$.
\end{theorem}

\begin{proof}
The projection takes $(\nu, \sigma_j)/\alpha'_{(\nu, \sigma_j)}$
to $\nu/\alpha'_{(\nu, \sigma_j)}$, and it is the convex hull
of $\{ 0 \}$ and these points is exactly the dual polytope 
to the on defined by $(d', \alpha')$.
\end{proof}

\begin{theorem}
\label{mirrorIsBundle}
(Assuming $a'$ has a lift as above) $X' = E'$ if and only if, 
for all $j$, every element in $V_j^\times$
defines a vertex of 
$\mathrm{conv}(\{ 0 \} \cup \bigcup_j \pi_j(V_j))$.
\end{theorem}

\begin{proof}
This simply states that the rows are the same, which is exactly
what we need.
\end{proof}

\subsection*{The dual superpotential}
\label{subsection-dualSuper}

Finally, we discuss the superpotential $W'$ on the dual.
First we discuss arbitrary functions on the total space, 
$E'$, of a split vector bundle over a toric variety $Y'$.

Let $Y'$ be a toric variety, $\{ D'_1 ,\cdots, D'_c \}$
a set of $T$-invariant divisors,
and $f$
a function on $E' := \mathrm{Tot}(\mathcal{O}_Y(-D'_1) \oplus \cdots \oplus \mathcal{O}_Y(-D'_c))$.
We want to know when $f$ comes from a global section of
$\mathcal{O}_Y(D'_1) \oplus \cdots \oplus \mathcal{O}_Y(D'_c)$.

The character group of  $E'$ is given by $M_{Y'} \oplus \Bbb{Z}\xi'_1 \oplus
\cdots \oplus \Bbb{Z}\xi'_c$.
Meromorphic global sections have the form $f=f_1 \xi'_1 + \cdots + f_c \xi'_c$, where the $f_j$'s have terms in $M_{Y'}$.
The key here is that $f$ is a linear form in the $\xi'$'s.
As we mentioned before in equation (\ref{regularCharacter}),
$\mu \in M_{E'}$ is regular iff $\mathrm{div}_{E'}(\mu) \geq 0$.

Now, assume that $X = \mathrm{Tot}(\mathcal{O}_Y(-D_1) \oplus \cdots \oplus \mathcal{O}_Y(-D_c))$
over a toric variety $Y$.  Also assume the dual toric variety $X(A',a')$
equals  $E'$.

\begin{lemma}
\label{mirrorSuperSection}
$W'$ comes from a global section exactly
when $\exists$ effective $T$-invariant divisors on $X$ such that 
$\tilde{D}_1, \cdots, \tilde{D}_c$ such that $D_j \sim \tilde{D}_j$,
and 
$\tilde{D}_1 + \cdots + \tilde{D}_c = -\kappa_X$.
Where $-\kappa_X$ is the canonical choice of anticanonical divisor representative
of equation (\ref{canon_anticanon}).
\end{lemma}

\begin{proof}
The terms of $W'$ are the rows of
$A=\mathrm{div}_{X}$.  The lemma is then  clear from the matrix of
$A^\tau$, see theorem \ref{bigA-formula}. 
\end{proof}

\begin{remark}
\label{remark:double_dual}
The dual to $(X', W', K')$ is computed from the dual of its linear data.  This
data is the same as that of $(X, W, K)$ except some rows of $\mathrm{mon}_W$
may have been deleted.  This means that the double dual of $(X, W, K)$
is $(X, W^-, K)$, where $X$ and $K$ are the same and $W^-$ is obtained from
$W$ by deleting some terms.  

If $(X, W, K)$ is obtained from a sigma model
then $(X, W^-, K)$ comes from a sigma model in the same ambient space and
the same bundle.  However the sigma models might differ by a complex 
deformation.  It is possible that this unfortunate 
discrepancy can be addressed with a slight modification of the 
dualization process.  We will explain this in the discussion 
section the end of this paper.
\end{remark}

\subsection*{An example: three points on $\Bbb{P}^1$}
A configuration on three points on $\Bbb{P}^1$
corresponds to a Landau-Ginzburg model $X = \mathrm{Tot}(\mathcal{O}_{\Bbb{P}^1}(-3))$,
$W$, and $K$.  If we identify $X$ with the blow up of $\Bbb{C}^2/\Bbb{Z}_3$ at the fixed point, $W$
is the pullback of a degree three homogeneous polynomial on $\Bbb{C}^2$ pushed forward to a
function on the quotient.

The character group is generated by elements that correspond to $u/v$ and $u^2 v$ on $\Bbb{C}^2$.
With respect to these coordinates we have
\begin{equation}
\mathrm{div}_X = 
\left[
\begin{array}{rr}
1 & 2 \\
-1 & 1 \\
0 & 1
\end{array}
\right],\ \  
\mathrm{and} \  \ 
\mathrm{mon}_W = 
\left[
\begin{array}{rr}
1 & 1 \\
-1 & 1 \\
-2 & 1 \\
0 & 1
\end{array}
\right].
\end{equation}
If we write $x$ and $y$ for the coordinates dual to these on $X'$, there are four
possibilities for $X'$ depending on the choice of $W$.
These are indicated by the 
polyhedral sets in figure \ref{fig:3pointsMirrors}.  If $\Lambda$ is a lift of $L$, examples
of values that give these
are
$\Im(\Lambda) = (\alpha', 0) = (0, 2, 5, 0), 
(0,3, 5, 0), (-1, -1, 0, 0)$, and 
$(-1, 0, -1, 0)$ 
respectively.

\begin{figure}[h]
\begin{center}
\includegraphics[scale=.5]{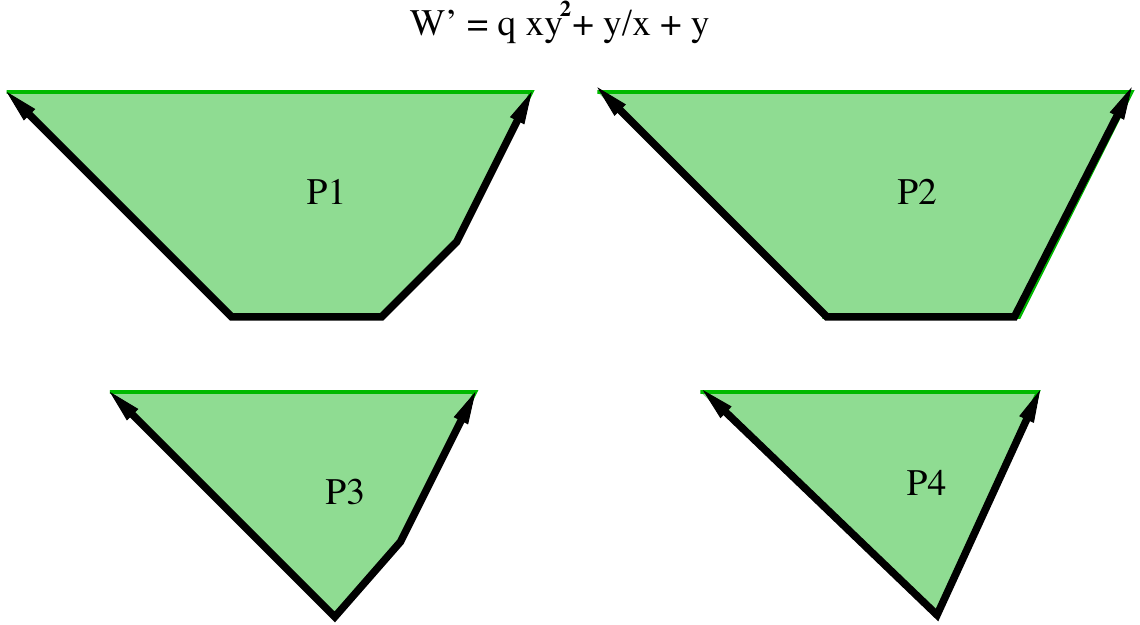}
\end{center}
\caption{Mirror possibilities for three points on $\Bbb{P}^1$.}
\label{fig:3pointsMirrors}
\end{figure}

The inward normals come from the rows of $\mathrm{mon}_W$.
The superpotential $W'$ above, has monomials
whose exponent vectors are the rows of 
$\mathrm{div}_X$. $q$ is determined $K$ by choosing
a lift of $K$ of the form $(\beta + i \alpha, 0 , 0)$ and
setting $q = \exp(-2 \pi  (\alpha + i \beta))$.  $W'$
is not linear in $y$, so it will not come from a global
section of $E'$.

\begin{equation}
d' = 
\left[
\begin{array}{r}
1 \\ 
-1 \\ 
-2
\end{array}
\right]
\end{equation}

The toric variety $Y'$ is given by 
the polyhedral set in $\Bbb{R}$ defined by
\begin{equation}
\left[
\begin{array}{r}
1 \\ 
-1 \\ 
-2
\end{array}
\right] \xi + \alpha' \geq 0 \ .
\end{equation}

For the polyhedral set P1, $(d', a')$ is kopasetic.
The other polyhedral sets can  give kopasetic data
at specific values.
For example P2 for $\alpha' = (0, 1, 2, 0)$, P3 for 
$\alpha' = (-1, -1, 0, -1)$, and P4 for $\alpha' = (0,0,0,0)$. 
One can easily 
check this by noting that the rows $d'$
that define facets, must be primitive for 
the map $k$ to exist.
However, we will only look at P1 since
these other cases do not
shed more light on the situation.

In this case 
\begin{equation}
\mathrm{div}_{Y'} = 
\left[
\begin{array}{r}
1 \\
-1
\end{array}
\right] \ ,
\end{equation}
and $k$ projects onto the first two basis vectors, so $k((0,2,5,0)) = (0,2)$.

For the divisors we have
\begin{equation}
D' = D_1' = \left[
\begin{array}{r}
1 \\
1 \\
1
\end{array}
\right] \ , \ \mathrm{and} \quad  
k(D_1') = \left[
\begin{array}{r}
1 \\
1
\end{array}
\right] \ .
\end{equation}
This means that $Y' = \Bbb{P}^1$, and $E'$ is the total space of $\mathcal{O}(-2)$.
The polytope corresponding to this toric variety is below. 
\begin{figure}[h]
\begin{center}
\includegraphics[scale=.5]{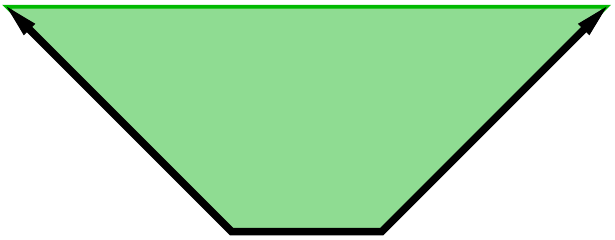}
\end{center}
\caption{Tot $\mathcal{O}(-2)$ over  $\Bbb{P}^1$.}
\label{fig:totO(-2)}
\end{figure}

This figure does not agree with P1.  This is predicted by theorem \ref{mirrorIsBundle}
since all facets are present, but for $Y'$ not all are needed.
One interesing
thing that is apparent in this example and happens in general is that the elements of 
$V_j^\times $ that are not vertices of $\mathrm{conv}(\pi_j(V_j^\times))$
serve to partially compactify $E'$.  Furthermore, both $E'$ and $X'(A',\Im(K'))$ are 
always local Calabi-Yau (i.e. have trivial canonical class).

\smallskip

\section{Comparison: Givental and Hori-Vafa}

\subsection{Givental}
\label{GIVequalsFTLG}

Let $Y$ be a $n$-dimensional smooth complete toric variety with $T$-invariant Cartier
divisors $R_Y = \{\rho_1, \cdots, \rho_r \}$.  
Recall, \begin{equation}\mathrm{coker}(\mathrm{div}_Y) = \mathrm{H}^2(Y; \Bbb{Z})\end{equation}
and is torsion free.

Let $D_1, \cdots, D_c$
be effective $T$-invariant Cartier divisors. Set
$\mathcal{V} := \mathcal{O}_Y(D_1) \oplus \cdots \oplus \mathcal{O}_Y(D_c)$,
as usual, and write $X$ for the total space of the dual.
Assume that there is no point in $Y$ at which all global sections of $\mathcal{V}$
vanish.

Let $p_1, \cdots, p_{r-n}$ be a positive basis for $\mathrm{H}^2(Y; \Bbb{Z})$,
and $\omega$ a symplectic form on $Y$ with cohomology class 
\begin{equation}[\omega] = \sum_i t_i p_i . \end{equation}
Denote $\mathrm{exp}(t_i)$ by $Q_i$.
Set \begin{equation}[\rho_\upsilon] = \sum_i m_{i\upsilon} p_i \quad  \mathrm{and} \quad  [D_j] = \sum_i d_{ij} p_i . \end{equation}
So a representative for $t$ is $\hat{t}$ with $\sum \hat{t}_\upsilon m_{i \upsilon} = t_i$.

In this notation, $[-] \colon   \Bbb{Z}^{R_X} \rightarrow \mathrm{H}^2(Y; \Bbb{Z})$
is given by
\begin{equation}
\left[
\begin{array}{ccc}
\mathcal{M} & | & -D
\end{array}
\right]
 . \end{equation}

Define \begin{equation}F(x,y) := x_1 + \cdots + x_r + y_1 + \cdots + y_c\end{equation}
as a function on  \begin{equation}H':=\{(x,y)\ |\   \prod_{\upsilon=1}^r x_\upsilon^{m_{i\upsilon}} = 
Q_i \prod_{j=1}^c y_j^{d_{ij}}\ , \ i = 1, \cdots, r-n \}
\quad \subseteq (\Bbb{C})^{r+c}_{x,y} \times (\Bbb{C^\times})^{r-n}_Q . \end{equation}

According to \cite[pg. 45]{ellipticGivental}, relations in the quantum cohomology
of the zero locus $(\mathrm{w})_0$ of a generic section
$\mathrm{w}$ of
$\mathcal{V}$
can be described by differential operators annihilating
integrals of the form
\begin{equation}
\label{GivMirror}
\int \mathrm{exp}(F/\hbar) \prod_{\upsilon=1}^r  x_\upsilon^{\lambda_\upsilon/\hbar} \prod_{j=1}^c y_j^{-\lambda'_j/\hbar} \ \ 
\mathrm{dlog}(x) \wedge \mathrm{dlog}(y) / \mathrm{dlog}(Q)
\end{equation}
along certain cycles in the fiber over $Q$, $H'_Q$.   Here $\mathrm{dlog}(x) = 
\mathrm{dlog}(x_1) \wedge \cdots \wedge \mathrm{dlog}(x_r)$,  
$\mathrm{dlog}(y) = \mathrm{dlog}(y_1) \wedge \cdots \wedge \mathrm{dlog}(y_c)$, and 
$\mathrm{dlog}(Q) = \mathrm{dlog}(Q_1) \wedge \cdots \wedge \mathrm{dlog}(Q_{r-n})$.

Note that the quotient of differential forms is unambiguous  since
any form annihilated when multiplied by $\mathrm{dlog}(Q)$
restricts to zero along $Q =$ constant.

Let $(X,W,K)$ be defined by $\mathrm{w}$ and $K = -it/2\pi \in \mathrm{H}^2(Y; \Bbb{C}/\Bbb{Z})$.
Denote the dual Landau-Ginzburg model by $(X',W',K')$.
We will compare $(X',W', K')$ to  $(H_Q, F, 0)$.
The main step is the theorem below.  We restate our assumptions about $Y$ and $\mathcal{V}$ 
for clarity.

\begin{theorem}
\label{GivEQAff}
Assume  $Y$ is smooth and there is no point that lies
on $(\sigma)_0$ for all $\sigma \in \Gamma(Y, \mathcal{V})$.
Then $M_{X'}^+$ is generated by the terms $\xi' \in \Xi'$
of $W'$, and there is an isomorphism for any $Q$,
\begin{equation}
\mathrm{Spec}(\Bbb{C}[M_{X'}^+]) \tilde{\longrightarrow} H_Q' \subset \Bbb{C}^{r+c} \ .
\end{equation}
If enumerate the terms of $\Xi'$ in the same way as the divisors of $X$, 
$R_X = \{ \rho_\upsilon \}_{\upsilon=1}^r \cup \{ X_j \}_{j=1}^c$,
the map is given by $x_\upsilon \mapsto \hat{q}_\upsilon \xi'_\upsilon$
and $y_j \mapsto \xi_j$.
Here $\hat{q}_\upsilon := \mathrm{exp}(\hat{t}_\upsilon)$.
\end{theorem}

\begin{proof}
Since $Y$ is complete, applying $\mathrm{Hom}( - , \Bbb{Z})$ 
to equation
(\ref{completeCase})
yields the exact sequence
\begin{equation}
0 \rightarrow \mathrm{H}^{2n-2}(Y; \Bbb{Z})^\vee \rightarrow \mathrm{Hom}( \Bbb{Z}^{\Xi'}  , \Bbb{Z})
\rightarrow M_{X'} \rightarrow 0 . 
\end{equation}
The image of the standard basis vectors satisfy relations 
coming from the rows of 
\begin{equation}
\left[
\begin{array}{ccc}
\mathcal{M} & | & -D
\end{array}
\right]  . 
\end{equation}
Therefore, $H' = \mathrm{Spec}(\Bbb{C}[ \ \xi' \ | \ \xi' \in  \Xi' \ ])$.

It remains to check that the terms of $W'$ generate the semigroup $M^+_{X'}$.
All of the terms are global, so they lie in $M^+_{X'}$. 

The terms of $W'$
are the same as the inward normals of the polytope defining the original 
toric variety $X$, and their $\Bbb{R}_{\geq 0}$-span intersected with $M_{X'}$ 
is $M_{X'}^+$.  One way to see this is to  
consider a 1-parameter
subgroup $\nu$ of $T$ that does not lie in the fan of $X$.   
Consider the fan of the projective bundle compactifying $X$ to a 
complete toric variety.
Then it is clear that $\lim_{\lambda \rightarrow 0} \nu(\lambda)$ lies 
in on the divisor at infinity. All of the monomials defining $W$ 
have a pole along this divisor and 
there is no chance that all the terms of $W$  have $\nu(0)$
on their divisor of zeros since this would mean that the sections 
they correspond to would all vanish at the point $\nu(0)$
projects to in $Y$.  Finally, this means that there is 
at least one term for which the limit at $\nu(0)$ is
infinity, and so $n$ is not in $M_{X'}^+$.

It remains to check positive integer multiples of these
normals pick up every integer point in the cone they define.
The cone they generate is naturally 
decomposed into the fan of $X$.  Since the fan is contained
in a half-space, we pick up all integral points with integral 
linear combinations of the normals if and only if $X$ is smooth.

The last observation needed to guarantee that we land in $H'_Q$
under the isomorphism above is that 
$Q_i = \prod_\upsilon \hat{q}_\upsilon^{m_{i \upsilon}}$.
\end{proof}

We see from the proof that if there is a point at which all 
global sections vanish, then $\mathrm{Spec}(\Bbb{C}[M_{X'}^+])$
is an affine open set of $H'_Q$.  If $Y$ is singular, then 
there is an \'etale map $\mathrm{Spec}(\Bbb{C}[M_{X'}^+]) 
\rightarrow H'_Q$.
In any case we have a map 
$\mathrm{Spec}(\Bbb{C}[M_{X'}^+]) 
\rightarrow H'_Q$.

The last thing to check is that $F$ pulls back to $W'$, but this is 
obvious from the definition of the isomorphism.

So we get the following theorem.

\begin{theorem}
There is a morphism $X' \rightarrow \Bbb{C}^{r+c}$ such the image is $H_Q$
and the function $F(x,y) = x_1 + \cdots + x_r + y_1 + \cdots + y_c$ pulls back to $W'$.
This morphism is naturally interpreted as the result of the K\"ahler degeneration of $X'$
under which $K' \leadsto 0$.
\end{theorem}

\begin{remark}
If $L$ already equals zero, then $X' = H_Q$.  The section corresponding to $L = 0$
of $\mathcal{V}$ is the $\mathrm{w}_{GHV}$ mentioned in the introduction.
\end{remark}

\subsection{Hori and Vafa}
\label{HVequalsFTLG}
In this section we show that the mirror used by Hori and Vafa is the same as
the one used by Givental.

To get started, observe the integral above in equation (\ref{GivMirror}) used by Givental 
can be manipulated according to the following rule for forms with delta functions as coefficients:
\begin{equation}
\delta(z) \varphi
=
(\varphi/dz)|_{z=0}
 . \end{equation}

With repeated application of this rule we can write the integral of Givental on $(n+c)$-cycles in $H'_Q$ as
\begin{equation}
\int \mathrm{exp}(F/\hbar) \prod_{\upsilon=1}^r  x_\upsilon^{\lambda_\upsilon/\hbar} \prod_{j=1}^c y_j^{-\lambda'_j/\hbar} \ \
\prod_{i=1}^{r-n} \delta(\mathrm{log}(Q_i) - t_i) \ \
\mathrm{dlog}(x) \wedge \mathrm{dlog}(y)
\end{equation}
over $(r+c)$-cycles in $\Bbb{C}^{r+c}$.

Now we consider the construction of Hori and Vafa.
As before  $Y$ is an $n$-dimensional toric variety,  This time it is  obtained as a quotient $\Bbb{C}^r // (\Bbb{C}^\times)^{r-n}$.
This can always be done as shown by Cox \cite{CoxHomog}.
Let $\frak{m}_{i \upsilon}$ be the weight of the action of $i^{\mathrm{th}} \ \Bbb{C}^\times$ on the $\upsilon^{\mathrm{th}}$
$\Bbb{C}$.  Let $G_1, \cdots, G_c$ be multi-homogeneous polynomials on $\Bbb{C}^r$ where $\frak{d}_{i j}$ is
the degree of the $j^\mathrm{th}$ polynomial with respect to the $i^{\mathrm{th}}$ $\Bbb{C}^\times$.

Set $\frak{W}' = \sum_\upsilon \mathrm{exp}(- \frak{Z}_\upsilon) + \sum_j \mathrm{exp}(- Y_{j})$.  In \cite{HoriVafa}
the following integrals are considered:

\begin{equation} \int \prod_{\upsilon=1}^r \mathrm{d}\frak{Z}_\upsilon \prod_{j=1}^c
\mathrm{d}Y_{j}  \prod_{j=1}^c \mathrm{exp}(-Y_j) \prod_{i=1}^{r-n} \delta(\sum_{\upsilon=1}^r  \frak{m}_{i \upsilon} \frak{Z}_\upsilon - \sum_{j=1}^c \frak{d}_{i j} Y_j - \frak{t}_i)
\ \ \mathrm{exp}(-\frak{W}') . \end{equation}

\begin{remark}
The integral above is a slight modification of \cite{HoriVafa} equation (7.78)
which is supposed to be a generalization of equation (7.32) in that paper. However,
equation (7.78) does not have equation (7.32) as a special case.  On the other hand, the
integral here does.
\end{remark}

Set for $i = 1, \cdots, r-n$
\begin{equation}\mathrm{log}(\frak{Q}_i) := -\sum_{\upsilon=1}^r \frak{m}_{i \upsilon} \frak{Z}_\upsilon + \sum_{j=1}^c \frak{d}_{i j} Y_j . \end{equation}

Writing $-\frak{Z}_\upsilon = \mathrm{log}(x_\upsilon)$ and $-Y_j = \mathrm{log}(y_j)$ we obtain
\begin{equation} (-1)^{(r+c)}\int \prod_{\upsilon=1}^r \mathrm{dlog}(x_\upsilon) \prod_{j=1}^c
\mathrm{dlog}(y_{j})  \prod_{\upsilon=1}^r x_\upsilon^{-1} \prod_{j=1}^c y_j^0
\ \ \prod_{i=1}^{r-n} \delta(\mathrm{log}(\frak{Q_i})+\frak{t}_i) \ \mathrm{exp}(-F(x,y)) . 
\end{equation}

\begin{remark}
Observe that when setting  $\hbar =-1$, $\lambda_\upsilon = 1$, and $\lambda_j' = 0$ in Givental's integrals we have
\begin{equation}
\int \mathrm{exp}(-F(x,y)) \prod_{\upsilon=1}^r  x_\upsilon^{\lambda_\upsilon/\hbar} \prod_{j=1}^c y_j^{-\lambda'_j/\hbar} \ \
\prod_{i=1}^{r-n} \delta(\mathrm{log}(Q_i) - t_i) \ \
\mathrm{dlog}(x) \wedge  \mathrm{dlog}(y) 
 . \end{equation}
\end{remark}

The following theorem verifies that the integrals considered by
Hori and Vafa are exactly those
of Givental specialized as in the above remark.

\begin{theorem}
$\frak{m}_{i \upsilon} = m_{i \upsilon}$, $\frak{d}_{ij} = d_{ij}$, $\frak{t}_i = -t_i$, and $\frak{Q_i} = Q_i$
\end{theorem}

\begin{proof}
The weight matrix, $(\frak{m})_{i \upsilon}$, is simply the map on Lie algebras $\mathrm{d_{id}}(- \circ [-]) =
(fr \circ [-])^\tau
= [-]^\tau$  .
\end{proof}

\section{Comparison: Berglund-H\"ubsch}

Let $Y$ be the $n$-dimensional weighted projective space $\Bbb{P}(l_0, \cdots, l_n)$
and $X$ a Calabi--Yau hypersurface in $Y$.  If we set $d:=l_0 + \cdots +l_n$, $X$
is defined by a weighted homogeneous polynomial $G$ of degree $d$ in the variables
$x_0, \cdots, x_n$, where the degree of $x_i$ is $l_i$.  In \cite{BerglundHubsch:genConstruct}
Berglund and H\"ubsch
consider the situation in which  all degree $d$ monomials
except $x_0 \cdots x_n$ appear in the expansion of $G$.

They define
$P$ to be the matrix whose columns are the exponent vectors of the terms of $G$.
To be clear,
\begin{equation}
\left[ \begin{array}{cccc} l_0 & \cdots & l_n & -d \end{array} \right] \cdot
\left[
\begin{array}{cc}
P & \begin{array}{c} 1 \\ \vdots  \end{array} \\
 \begin{array}{cc}1 & \cdots  \end{array} & 1
\end{array}
\right]= 0 \ .
\end{equation}

They then  define  $\hat{l}_0, \cdots, \hat{l}_n$, and $\hat{d}$  by
\begin{equation}
\left[ \begin{array}{cccc} \hat{l}_0 & \cdots & \hat{l}_n & - \hat{d} \end{array}
\right] \cdot
\left[
\begin{array}{cc}
\hat{P} & \begin{array}{c} 1 \\ \vdots  \end{array} \\
 \begin{array}{cc}1 & \cdots  \end{array} & 1
\end{array}
\right]= 0 \ ,
\end{equation}
where $\hat{P} := P^\tau$. 

They then obtain the mirror $\hat{X} \subset \Bbb{P}(\hat{l}_0, \cdots, \hat{l}_n)$ defined
by a homogeneous degree $\hat{d}$ polynomial $\hat{G}$
in the variables $\hat{x}_0, \cdots, \hat{x}_n$,
where now the degree of $\hat{x}_i$ is $\hat{l}_i$.
All degree $\hat{d}$ terms appear in $\hat{G}$ except $\hat{x}_0 \cdots \hat{x}_n$.
$\hat{P}$ is the matrix whose columns are the exponent vectors of the terms of $\hat{G}$.
Note that the Berglund and H\"ubsch set all coefficients to one.

Observe that 
\begin{equation}
\left[
\begin{array}{cc}
P & \begin{array}{c} 1 \\ \vdots  \end{array} \\
\begin{array}{cc}1 & \cdots  \end{array} & 1
\end{array}
\right]
\end{equation}
factors as $A \cdot B^\tau$, where $A$ and $B$ are obtained from
the sigma model/Landau--Ginzburg model correspondence applied to
a degree $d$ hypersurface $X_\epsilon$ in $\Bbb{P}(l_0, \cdots, l_n)$
whose equation uses all degree $d$ monomials.

Similarly,  
\begin{equation}
\left[
\begin{array}{cc}
\hat{P} & \begin{array}{c} 1 \\ \vdots  \end{array} \\
 \begin{array}{cc}1 & \cdots  \end{array} & 1
\end{array}
\right]
\end{equation}
 factors as $\hat{A} \cdot \hat{B}^\tau$

The key point is that  $\hat{A}$ and $\hat{B}$
are obtained from
the sigma model/Landau--Ginzburg model correspondence applied to
a degree $\hat{d}$ hypersurface $\hat{X}_{\hat{\epsilon}}$ in $\Bbb{P}(\hat{l}_0, \cdots, \hat{l}_n)$
whose equation uses all degree $\hat{d}$ monomials.
Furthermore, since $\left[ \begin{array}{cccc} \hat{l}_0 & \cdots & \hat{l}_n & - \hat{d} \end{array}
\right]$ is the cokernel of $B$ it is immediate that $\hat{A} = B$ and $\hat{B} = A$.
Thus $X_\epsilon$ and $\hat{X}_{\hat{\epsilon}}$ are in mirror families.

Finally, these families contain the hypersurfaces considered by
Berglund and H\"ubsch, and in fact one can simply apply  (non-toric) automorphisms
to $\Bbb{P}(l_0, \cdots l_n)$ and $\Bbb{P}(\hat{l}_0, \cdots , \hat{l}_n)$
to bring $X$ and $\hat{X}$ to the form $X_\epsilon$ and $\hat{X}_{\hat{\epsilon}}$.
The coefficients can be set to one for appropriate choices of $L$ and $L'$.

\section{Comparison: Batyrev-Borisov}

In order to describe  precisely the construction of Batyrev and Borisov,
we require several more  definitions.

\begin{definition}
\label{DefReflexive}
\cite{dualBatyrev}
A lattice polytope $P \subset M_\Bbb{R}$ with $0 \in \mathrm{int}(P)$ 
is called {\it reflexive} if $P^\circ$ also a
lattice polytope. It is clear that $P^\circ$ is also reflexive.
\end{definition}

\begin{remark}
In light of corollary \ref{redundantPolar},
reflexivity simply means that $P = \{ \mu \in M_\Bbb{R} \ | \ 
\nu(\mu) + 1 \geq 0, \ \mathrm{where\ }
\nu \ \mathrm{is\ a\ primitive\ inward\ normal} \}$.
This is the same as saying that $P$ is the anticanonical polytope of a toric variety.
In particular, such a $P$ is be the anticanonical polytope of the toric variety it defines.
\end{remark}

\begin{definition}
\label{DefNefPartition}
\cite{dualNefBorisov}
If $P$ is a reflexive polytope, and $\Sigma_P$ its inward normal fan,
a {\it nef-partition} of $\mathrm{vert}(P)$
is a partition $\{ \Bbb{E}_1, \cdots, \Bbb{E}_c \}$ of $\mathrm{vert}(P)$
such that there exist integral convex $\Sigma_P$-piecewise linear functions
$\phi_1, \cdots, \phi_c$ on $M_\Bbb{R}$ satisfying $\phi_i(e_j) = \delta(i-j)$
for $e_j \in \Bbb{E}_j$.
\end{definition}

\begin{remark}
It is standard procedure in the theory of toric varieties to equate
an integral convex function $\phi_j$ with a $T$-invariant Cartier divisor $D_j$ on
$X(\Sigma)$. In our case, $D_j$  is the sum of divisors who corresponding
facet of $P$ has its  inward normal in $\Bbb{E}_j$.
\end{remark}

If $Y$ is a toric variety defined by a reflexive polytope $P$,
it will be more natural to consider nef-partitions of $P^\circ$.

\label{relation-BB}
\begin{definition}
\cite{dualNefBorisov}
Let $P^\circ$ be a reflexive polytope and $\{\Bbb{E}_1, \cdots, \Bbb{E}_c \}$
a nef partition of $\mathrm{vert}(P^\circ)$.
Define $\nabla_j := \{ \mu \in M_\Bbb{R} \ | \nu(\mu) \geq - \phi_j(\nu), \ \forall \nu \in N_\Bbb{R} \}$,
and $(P^*)^\circ :=  \mathrm{conv}(\nabla_1 \cup \cdots \cup \nabla_c)$.
Set $\nabla^\times_j = \mathrm{conv}(\nabla_j \setminus \{0\})$, and $\Bbb{E}^*_j = \mathrm{vert}(\nabla^\times_j)$.
Then $(P^*)^\circ$ is reflexive, and $\{ \Bbb{E}^*_1, \cdots, \Bbb{E}^*_c \}$
is a nef partition of $\mathrm{vert}((P^*)^\circ)$
\cite[Prop 3.4]{dualNefBorisov}. $\{ \Bbb{E}^*_1, \cdots, \Bbb{E}^*_c \}$ is called
the {\it dual nef partition} to $\{ \Bbb{E}_1, \cdots, \Bbb{E}_c \}$.
\end{definition}

\begin{definition}
We will write $P^*$ for $((P^*)^\circ)^\circ$.
\end{definition}

\begin{definition}
\label{definition:BBmirror}
Let $Y$ be a toric variety defined by the reflexive polytope $P$, and let
$\{\Bbb{E}_1, \cdots, \Bbb{E}_c \}$
a nef partition of $\mathrm{vert}(P^\circ)$ with corresponding divisors 
$\{ D_1, \cdots, D_c \}$. 
From this define $Y^*$ to be the toric 
variety defined by $P^*$,  $\{ \Bbb{E}^*_1, \cdots, \Bbb{E}^*_c \}$ 
the dual nef partition, and $\{D_1^*, \cdots, D_c^* \}$ the corresponding
divisors on $Y^*$.  Denote $\mathcal{V} := \mathcal{O}_Y(D_1) \oplus \cdots \oplus \mathcal{O}_Y(D_c)$
and $\mathcal{V}^* := \mathcal{O}_{Y^*}(D^*_1) \oplus \cdots \oplus \mathcal{O}_{Y^*}(D^*_c)$. 
The complete intersections in $Y$ given by the vanishing of global sections of $\mathcal{V}$ are 
{\it Batyrev-Borisov mirror} to the complete intersections in $Y^*$ 
given by vanishing of global sections of $\mathcal{V}^*$.
\end{definition}

\subsection*{Application of section \ref{section:structure}}
To make a comparison between our duality with this construction, we revisit the
results of section \ref{section:structure}, in particular theorem \ref{mirrorIsBundle}. 
Now let $Y$ be an arbitrary toric variety.

\begin{definition}
A collection $\{D_1, \cdots, D_c \} \subseteq \mathbb{Z}^{R}$
is called {\it Givental} if $D_j > 0$ for all $j$ and $-\kappa_Y - \sum_j D_j \geq 0$.
If $[-\kappa_Y] = [\sum_ j D_j]$, we say the collection is {\it Calabi--Yau}.
\end{definition}

\begin{definition}
Denote $\mathrm{conv}(P_{D_j}^\times \cap M_Y)$ by $C_j$, and 
$\mathrm{conv}(\cup_j(P_{D_j}^\times \cap M_Y))$ by $C$.
\end{definition}

It is always true that  
$\mathrm{vert}(C) \subseteq \cup_{j} \mathrm{vert}(C_j)$,
and this leads to the follwing defintion.

\begin{definition} 
A {\it nef sub-partition} is a Givental collection such that
$\mathrm{vert}(C) = \cup_{j} \mathrm{vert}(C_j)$,
and $C_j
\neq \varnothing$ for all $j$.
\end{definition}

The following theorem allows us to express Borisov's notion of 
nef partition in these terms.

\begin{theorem}
\label{NefEqualsNef}
If $Y$ defined by a reflexive polytope $P = P_{-\kappa_Y}$,
a Calabi--Yau nef sub-partiton is the same as a nef partition of
$\mathrm{vert}((P_{-\kappa_Y})^\circ)$.
\end{theorem}

\begin{proof}
Given a Calabi--Yau nef sub-partition, define $\Bbb{E}_j =$ set of primitive
inward normals to the
affine half-spaces  corresponding the the toric divisors appearing as
summands of $D_j$.
Now define $\phi_j(\nu) = -\mathrm{inf}_{\mu \in P_{D_j}}(\nu(\mu))$. Then the 
definition of $P_{D_j}$ guarantees the  $\phi_i(e_j) = \delta(i-j)$ as needed. 
So $\{\Bbb{E}_1 , \cdots, \Bbb{E}_c\}$ is a
nef-partition (definition \ref{DefNefPartition}).

Now assume we have a nef-partition of $\mathrm{vert}((P_{-\kappa_Y})^\circ)$.  Since $Y$ is
Fano, these vertices
are primitive inward normals to the affine hypersurfaces that correspond to toric
divisors of $Y$.
Define $D_j :=$ sum of toric divisors whose primitive inward normal is in $\Bbb{E}_j$.  This
trivially gives a Givental, Calabi--Yau collection.  It remains to check the nef sub-partition
condition.  This is proved in  \cite[Prop 3.4]{dualNefBorisov}.
\end{proof}

\begin{definition}
Let $Y$ be a Fano toric variety, and $\{ D_1, \cdots, D_c \}$
a Calabi-Yau nef sub-partition  and 
$\mathcal{V} = \mathcal{O}_Y(D_1) \oplus \cdots \oplus \mathcal{O}_Y(D_c)$.
Identify integral points $P_{D_j}$
with a basis for global sections of $\mathcal{O}_Y(D_j)$. Define $\mathrm{w}_{BB}$
to be the global section of $\mathcal{V}$ given by 
\begin{equation}
\mathrm{w}_{BB} := \sum_j \sum_{0 \neq \sigma  \in P_{D_j} \cap M_Y} \exp(-2 \pi) \sigma \ .
\end{equation}
\end{definition}

Finally, the following theorem shows that the construction of Batryev and Borisov
is a special case of our duality  (refer to definition \ref{definition:BBmirror} 
for the definition of $\mathcal{V}^*$.)

\begin{theorem}
\label{BBequalsFTLG}
Let $Y$ be an $n$ dimensional  Fano toric variety, $\{ D_1, \cdots, D_c \}$
a Calabi-Yau nef sub-partition,  and 
$\mathcal{V} = \mathcal{O}_Y(D_1) \oplus \cdots \oplus \mathcal{O}_Y(D_c)$.
Let $X$ be the total space of the dual, $\mathcal{V}^\vee$, and 
$W_{BB} \colon X \rightarrow \Bbb{C}$ the superpotential defined by 
$\mathrm{w}_{BB}$.  For arbitrary $K \in A_{n-1}(Y)$, the Landau-Ginzburg 
model $(X', W', K')$, dual to $(X, W, K)$, has 
$X' = \mathrm{Tot}((\mathcal{V}^*)^\vee)$ and $W'$ comes from a global section of 
$\mathcal{V}^*$.
\end{theorem}

\begin{proof}
The choice of $W_{BB}$ means that on the dual $K' = (\alpha',0)$,
with $\alpha'_{(\nu,\sigma_j)} = 1$ for all $\nu \in P^\times_{D_j} \cap M_Y$.
For this value $(d', a')$ is easily seen to be kopasetic.  Now the
set $V_j^\times$ in definition \ref{definition:V_j} is $(P_{D_j}^\times \cap M_Y) \times \{ \sigma_j \}$.
The projection is $(P_{D_j}^\times \cap M_Y) \subset (M_Y)_\Bbb{R}$.
The nef sub-partition condition and theorem \ref{mirrorIsBundle} 
guarantee that $X' = E' = \mathrm{Tot}((\mathcal{V}')^\vee)$ over $Y'$.

If we denote the polytope defining $Y$ by $P$, it remains to show that $Y'$ 
is defined by $P^*$, and $D'_j = D^*_j$. The first fact follows from 
checking that $P_{D'_j} = \Delta_j$, where $\Delta_j = \mathrm{conv}(\{0\} \cup \Bbb{E}_j)$.  
Since $P' = P_{-k_{Y'}} = \sum_j P_{D'_j} $ from the standard theory of toric 
varieties, and $P^* = \sum_j \Delta_j$ by \cite[Prop 3.2]{dualNefBorisov}.
On the other hand, the second fact follows from $P_{D_j} = \nabla_j$
which implies $\mathrm{vert}(\nabla_j^\times) = \mathrm{vert}(C_j)$
and so $\mathbb{E}^*_j$ is made up of exactly the vertices of $(P')^\circ$
defining $D'_j$.

$P_{D_j} = \nabla_j$: $\mu \in \nabla_j \iff \nu(\mu) + \phi_j(\nu) \geq 0 , \ \forall \nu \in (N_Y)_\Bbb{R}$.
Since $Y$ is complete, we can find $\lambda \geq 0$ such that
$\nu = \sum_{e \in \mathrm{vert}(P^\circ)} c_e e$. Plugging this in above gives
$\sum_e c_e (e(\mu) +  \phi_j(e)) \geq 0$.  
Evaluating $\phi_j$ gives
$\sum_{e \not \in \Bbb{E}_j} c_e e(\mu) + \sum_{ e \in \Bbb{E}_j} c_e (e(\mu) +  1) \geq 0$.
This equation must hold for all choices $c_e \geq 0$, so we find 
$e(\mu) \geq 0$ for all $e \in \mathrm{vert}(P^\circ) \setminus \Bbb{E}_j$, and
 $(e(\mu) +  1) \geq 0$ for all $e \in \Bbb{E}_j$.  This is exactly the condition
that $\mu \in P_{D_j}$.

$P_{D'_j} = \Delta_j$: 
Using the equality of $P_{D_j}$ and $\nabla_j$ above
and the fact we only need to check non-zero vertices,
we can write $P_{D'_j}$ to be the $\nu$ satisfying 
$\nu(e^*) + 1 \geq 0$ for all 
$e^* \in \Bbb{E}^*_j$
and $\nu(e^*) \geq 0$ for all $e^* \in \Bbb{E}^*_i$ when $i \neq j$. 
$\Delta_j$ can be defined by 
$\{ \nu \in N_\Bbb{R} \ | \nu(\mu) \geq - \phi^*_j(\mu), \ \forall \mu \in M_\Bbb{R} \}$,
where $\phi^*_j(\mu) := -\inf_{\nu \in \nabla_j}(\nu(\mu))$ \cite[Cor 2.12 ]{dualNefBorisov}.
Now if we unravel this definition as we did in the previous paragraph,  we
arrive at the same conditions defining $P_{D'_j}$.

Finally, we check that the superpotential comes from a section by simply 
applying lemma \ref{mirrorSuperSection}.
\end{proof}

\section{An example: elliptic curves in $(\Bbb{P}^1)^2$}
\label{examples}

We consider  elliptic curves on $(\Bbb{P}^1)^2$.  This is a nice example since
many of the features of our duality  are exhibited and the dimension is low enough so that
we can actually draw the polytopes involved.  Also, since they are Calabi--Yau,
duality can be compared to the Batyrev--Borisov, Givental, and Hori--Vafa constructions.

Let $Z$ be an effective divisor with $Z \sim  D$, where $D$ is a $T$-invariant $(2,2)$ divisor on $Y = (\mathbb{P}^1)^2$.
$Z$ is an elliptic curve. This can be seen
using the adjunction formula, $\kappa_Z = (\kappa_Y + Z)|_Z$,
and the fact $\kappa_Y = (-2,-2)$.

In order to write down the Landau-Ginzburg model corresponding to $Z$ as in section \ref{correspondence},
we will use the inclusion $(\Bbb{C}^\times)^2_{x,y} \hookrightarrow Y$
defined by the point $([1:1],[1:1]) \in Y$ and the action
$(x,y) \cdot ([a:b],[c:d]) = ([a:xb],[c:yd])$.  This gives
the character group
\begin{equation}M_Y = \Bbb{Z} x \oplus \Bbb{Z} y ,\end{equation}
the group of $T$-invariant divisors
\begin{equation}\Bbb{Z}^{R_Y} = \Bbb{Z} \rho^x_0 \oplus \Bbb{Z} \rho^y_0 \oplus \Bbb{Z} \rho^x_\infty \oplus \Bbb{Z} \rho^y_\infty,  \end{equation}
the $1^{\mathrm{st}}$ Chow group
\begin{equation}A_{1}(Y) = \Bbb{Z} [\rho^x_0] \oplus \Bbb{Z} [\rho^y_0], \end{equation}
the character-to-divisor map
\begin{equation}
\mathrm{div}_Y
=
\left[
\begin{array}{cc}
1&0 \\
0&1 \\
-1&0 \\
0&-1
\end{array}
\right] ,
\end{equation}
and the cokernel of $\mathrm{div}_Y$,
\begin{equation}
[-]_Y =
\left[
\begin{array}{cccc}
1&0&1&0 \\
0&1&0&1
\end{array}
\right].
\end{equation}

Now for $E=\mathrm{Tot}(\mathcal{O}_Y(-2,-2))$ we use the formulas provided in
section \ref{AandB}.  First the  (character-to-divisor) $A$-side:

\begin{equation}
M = M_Y \oplus \Bbb{Z} \xi ,
\end{equation}
and
\begin{equation}\Bbb{Z}^{R} = \Bbb{Z}^{R_Y} \oplus \Bbb{Z} D .\end{equation}
Recall
\begin{equation}A_2(E) = A_{1}(Y) = \mathrm{H}^2(Y; \Bbb{Z}), \end{equation}
and we choose
\begin{equation} D = \left[ \begin{array}{c} 1 \\ 1 \\ 1 \\ 1 \end{array} \right] . \end{equation}
Note that this is the canonical choice for $-\kappa_Y$.
A quick check shows
\begin{equation} [D] = \left[ \begin{array}{c} 2 \\ 2 \end{array} \right]\end{equation}
as needed.

Putting this together gives
\begin{equation}
\label{ellipticDivE}
\mathrm{div}_E
=
\left[
\begin{array}{cc}
d_Y & D \\
0   & 1
\end{array}
\right]
=
\left[
\begin{array}{ccc}
\left[
\begin{array}{cc}
1&0 \\
0&1 \\
-1&0 \\
0&-1
\end{array}
\right] &  | & \left[ \begin{array}{c} 1 \\ 1 \\ 1 \\ 1  \end{array} \right] \\
\begin{array}{cc} 0 & 0 \end{array} &| & 1
\end{array}
\right] .
\end{equation}
Since $A_2(E)$ is torsion free the cokernel of $\mathrm{div}_E$, $[-]$,
is given by the  matrix
\begin{equation}
\label{equation:coker_div_E}
[-] =
\left[
\begin{array}{ccc}
[-]_Y & | & [-D]
\end{array}
\right]
=
\left[
\begin{array}{cc}
\left[
\begin{array}{cccc}
1&0&1&0 \\
0&1&0&1
\end{array}
\right]
&  \left[ \begin{array}{c} -2 \\ -2 \end{array} \right]
\end{array}
\right] .
\end{equation}
At this point we will not commit ourselves to a choice of $K$.

Now the (superpotential) B-side:
$N_Y = \Bbb{Z} \nu_1 \oplus \Bbb{Z} \nu_2$ with the obvious pairing.
So $N_E = N_Y \oplus \Bbb{Z} \phi$.  $P_{D}$ has integral points equal to the
characters with $\mathrm{div}_Y(\mu) + D \geq 0$.  This means $-1 \leq x,y \leq 1 $.
Thus
\begin{equation}\Xi =\{
\left[ \begin{array}{c} 0 \\ 1  \end{array} \right] ,
\left[ \begin{array}{c} 1 \\ 1  \end{array} \right] ,
\left[ \begin{array}{c} 1 \\ 0  \end{array} \right] ,
\left[ \begin{array}{c} 1 \\ -1  \end{array} \right] ,
\left[ \begin{array}{c} 0 \\ -1  \end{array} \right] ,
\left[ \begin{array}{c} -1 \\ -1   \end{array} \right] ,
\left[ \begin{array}{c} -1 \\ 0  \end{array} \right] ,
\left[ \begin{array}{c} -1 \\ 1  \end{array} \right] ,
\left[ \begin{array}{c} 0 \\ 0  \end{array} \right]
\} .\end{equation}
We can rewrite these vectors  as functions on $E$:
\begin{equation}\Xi =\{
y\xi ,
xy\xi ,
x\xi,
x\xi / y,
\xi / y,
\xi / xy,
\xi/x,
y\xi/x,
\xi
\} . \end{equation}
Now we can obtain $\mathrm{mon}$ by transposing the elements of $\Xi$
as mentioned after equation (\ref{Lequation}):
\begin{equation}
\mathrm{mon} =
\left[
\begin{array}{rrr}
0 & 1 & 1 \\
1 & 1 & 1 \\
1 & 0 & 1 \\
1 & -1 & 1 \\
0 & -1 & 1 \\
-1 & -1 & 1 \\
- 1 & 0 & 1 \\
-1 & 1 & 1 \\
0 & 0 & 1
\end{array}
\right] .
\end{equation}

Applying duality we have  $A = B' = \mathrm{div}_E$ and $B = A'= \mathrm{mon}$.

\subsection*{$K'$ = ``Givental choice''}
At this point we make a choice for $K' = L$ so that we can investigate the geometry of $X(A',a')$.
First choose $K' = 0$.  This leads to the Givental mirror and we know it is described by 
reading the rows off the cokernel of $A = \mathrm{div}_{X=E}$ which appears in equation 
(\ref{equation:coker_div_E}):
\begin{equation}
\{ (x_1, x_2, x_3, x_4, y_1) \in \Bbb{C}^5 \ | \ x_1 x_3 = Q_1 y_1^2, \mathrm{and} \  x_2 x_4 = Q_2 y_1^2   \},
\end{equation}
with the superpotential 
\begin{equation}
W' = x_1 + x_2 + x_3 + x_4 + y_1 \ .
\end{equation}
Indicating coordinates dual to $x, y$ and $\xi$ with primes, the identification of the Givental mirror with $X(A', 0)$ sets
\begin{equation}
x_1 =  x' \xi', \ \ x_2 =  y' \xi', \ \ x_3 = Q_1 \xi'/x', \ \ x_4 = Q_2 \xi'/y', \ \ \   \mathrm{and} \ \ \   y_1 = \xi' \ .
\end{equation}
The superpotential becomes 
\begin{equation}
W' = (x' + y' + Q_1 / x' + Q_2 / y'  + 1) \xi'
\end{equation}
on $X(A', 0) = \mathrm{Spec}(\Bbb{C}[x' \xi', y' \xi', \xi'/x', \xi'/y', \xi' ])$.

\subsection*{$K'$ = ``Batyrev-Borisov choice''} Now observe what happens for the ``Batyrev--Borisov'' choice $a' = (-\kappa',0)$.
In terms of the basis for $\Bbb{Z}^{\Xi}$,
\begin{equation}
L_{BB} = ia' = i[
y\xi +
xy\xi +
x\xi +
x\xi / y +
\xi / y +
\xi / xy +
\xi/x +
y\xi/x
]
.
\end{equation}
So $(\alpha',0) = (y\xi +
xy\xi +
x\xi +
x\xi / y +
\xi / y +
\xi / xy +
\xi/x +
y\xi/x, 0)$ gives a representative for $a'$ with $\alpha' = -\kappa  \in \Bbb{Z}^{\Xi^\times}$as desired.

The polytope of $X(A',a')$ is the set of solutions
$\left[ \begin{array}{c}n_1 \\ n_2 \\ p \end{array} \right] \in N_E = \Bbb{Z} \nu_1 \oplus \Bbb{Z} \nu_2 \oplus \Bbb{Z} \phi$ to
\begin{equation}
\left[ 
\begin{array}{c} 
n_1 \\ 
n_2 \\ 
p 
\end{array} 
\right] 
+ 
\left[ 
\begin{array}{c} 
\alpha' \\ 
0 
\end{array} 
\right]
=
\left[
\begin{array}{c}
0 \\
1 \\
1 \\
1 \\
0 \\
-1 \\
-1 \\
-1 \\
0
\end{array}
\right] 
n_1 +
\left[
\begin{array}{c}
 1 \\
 1 \\
 0 \\
 -1 \\
 -1 \\
 -1 \\
 0 \\
 1 \\
 0
\end{array}
\right] 
n_1 +
\left[
\begin{array}{c}
1 \\
1 \\
1 \\
1 \\
1 \\
1 \\
1 \\
1 \\
1
\end{array}
\right] 
p +
\left[
\begin{array}{c}
1 \\
1 \\
1 \\
1 \\
1 \\
1 \\
1 \\
1 \\
0
\end{array}
\right] 
\geq 0 .
\end{equation}

A careful check shows that inequalities corresponding to rows 1, 3, 5, and 7 above
are unnecessary.
Removing them yields the system of inequalities
\begin{equation}
\left[
\begin{array}{c}
1 \\
1 \\
-1 \\
-1 \\
0
\end{array}
\right] n_1 +
\left[
\begin{array}{c}
 1 \\
 -1 \\
 -1 \\
 1 \\
0
\end{array}
\right] n_1 +
\left[
\begin{array}{ccc}
1 \\
1 \\
1 \\
1 \\
1
\end{array}
\right] p +
\left[
\begin{array}{ccc}
1 \\
1 \\
1 \\
1 \\
0
\end{array}
\right] \geq 0 .
\end{equation}
It follows
\begin{equation}
\mathrm{div}_{X(A',a')} =
\left[
\begin{array}{rrr}
1 &1 &1 \\
1 &-1 &1 \\
-1 &-1 &1 \\
-1 &1 &1 \\
0 &0 &1
\end{array}
\right] ,
\end{equation}
and
$k(\alpha') =\left[
\begin{array}{c}
1\\
1\\
1\\
1\\
0
\end{array}
\right]
$ as in (as in  equation (\ref{Gmap})).

Comparing $\mathrm{div}_{X(A',a')}$ and $k(\alpha')$ to the
formula in equation (\ref{preAformula}), or more generally equation
(\ref{bigA-formula}), we see that $X(A', a')$ is the total space, $E'$, of
the canonical bundle over a variety $Y'$ with
\begin{equation}
\mathrm{div}_{Y'} =
\left[
\begin{array}{rr}
1 &1  \\
1 &-1  \\
-1 &-1  \\
-1 &1
\end{array}
\right] .
\end{equation}
Furthermore, $-\kappa_{Y'}$ pulls back to $k(\alpha')$.

After looking at the equation $\mathrm{div}_{Y'} \left[ \begin{array}{c} n_1 \\ n_2 \end{array} \right] + \kappa_{Y'} \geq 0$,
one finds $Y'$ is determined by the reflexive polytope $P_{-\kappa_{Y'}}$ that is given by
$\pm n_1 \pm n_2 + 1 \geq 0 \subseteq N_Y = \Bbb{Z}\nu_1 \oplus \Bbb{Z}\nu_2$:

\begin{figure}[ht]
\begin{center}
\quad \quad \ \
\includegraphics[scale=.75]{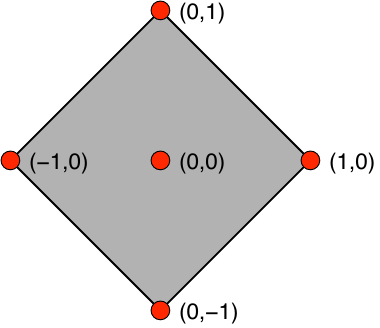}
\end{center}
\caption{Polytope for $Y'$ under the anticanonical choice for  $K'$.}
\label{fig:diamondPolytope}
\end{figure}

Note that this is dual to the polytope of $P_{-\kappa_Y}$ ( $P_{-\kappa_Y}$ is a square whose  sides have length = 2
and is  
centered at the origin of $M_{\Bbb{R}}$).

This is predicted by theorem \ref{BBequalsFTLG}.
To check that a function $W'$ on $X(A',a')$ defined by any choice of $L'$ comes
from a section of $\mathcal{O}_{Y'}(-\kappa_{Y'})$, one only needs to look at $\mathrm{div}_E$
in equation (\ref{ellipticDivE}) and notice that each row is the transpose of an element of
$P_{-\kappa_{Y'}} \times \{ \sigma \} \subset N_E$.  
Note the toric variety associated to this polytope is not smooth.

Also note that one can easily check that the $X(A',K'_{BB})$ can be obtained from Givental's mirror
by blowing up the scheme defined by $\xi' =0$.

\subsection*{$K'$ ``very ample''}
If we make a different choice for $K'$, for instance
\begin{equation}K'
= i[\left[ \begin{array}{c}
2\\
3\\
2\\
3\\
2\\
3\\
2\\
3\\
0
\end{array} \right] ] \end{equation}
A similar check shows that with this choice
$A'=\mathrm{div}_{X(A',a')}$,
and $X(A',a')$ is the total space of
the canonical bundle of the toric variety associated with the  stop-sign polytope:

\begin{figure}[ht]
\begin{center}
\ \quad \quad \quad
\includegraphics[scale=0.75]{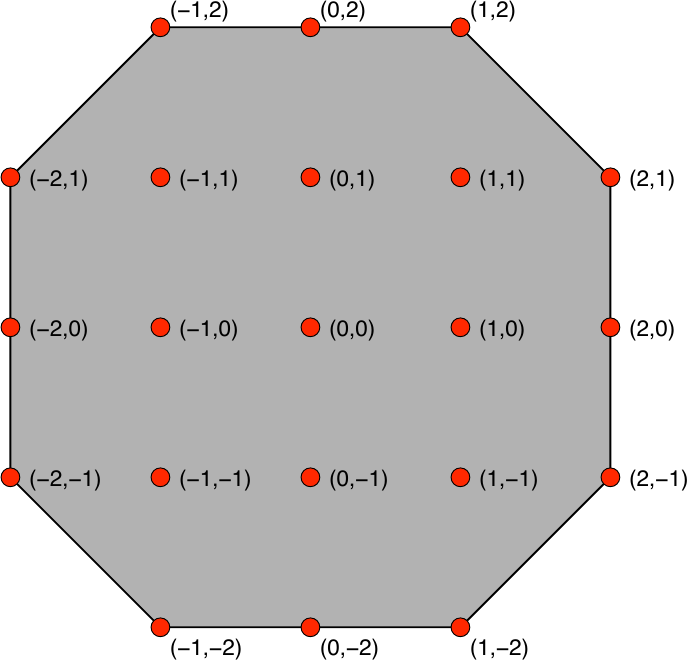}
\end{center}
\caption{Polytope for $Y'$ under a very ample choice for $K'$.}
\label{fig:octogonPolytope}
\end{figure}

This toric variety is smooth, and is in fact a crepant resolution
of the variety given by the diamond polytope. However, although
$W'$ comes from a global section of the anticanonical bundle,
it is not generic.

It is helpful to think of all three of these  varieties as both living in the
family of K\"ahler deformations
\begin{equation}
K'_{(s,t)}  = i t( s [\left[ \begin{array}{c}
2\\
3\\
2\\
3\\
2\\
3\\
2\\
3\\
0
\end{array} \right] ] +
i(1-s)[\left[ \begin{array}{c}
1\\
1\\
1\\
1\\
1\\
1\\
1\\
1\\
0
\end{array} \right] ])
.
\end{equation}

\section{Discussion}

The duality introduced here provides a simple 
framework under which all existing mirror candidate 
constructions 
for complete intersections fall.  It is also 
holds some promise for use in homological 
mirror symmetry since it often provides partial
resolutions when the mirror according to other
constructions is singular.

The main short coming in the construction is the apparent
asymmetry that sometimes happens under double dualization.
This was mentioned in remark \ref{remark:double_dual}.
It was pointed out to us by Lev Borisov that a possible
fix would be to include in the dualization procedure
a pair of (formal) automorphisms of  
$F_A \colon 
\mathrm{coker}(A)_{\Bbb{C}/\Bbb{Z}}
\rightarrow 
\mathrm{coker}(B')_{\Bbb{C}^\times}$, 
and
$F_B \colon 
\mathrm{coker}(B)_{\Bbb{C}^\times}
\rightarrow 
\mathrm{coker}(A')_{\Bbb{C}/\Bbb{Z}}$.

In the construction as defined here we are simply using the maps
$F_A = \exp(2\pi i - )$ and $F_B = \frac{1}{2\pi i}\log(-)$.
These would take $K$ to $\exp(2 \pi i L')$ and $\exp(2 \pi i L)$
to $K'$.  We could replace these with arbitrary maps, so long
as $F_{B'} \circ F_A = \mathrm{Id}$, and $F_B \circ F_{A'} = \mathrm{Id}$,
and all the results in the paper would still hold.

The ``correct'' maps $F_A$ and $F_B$ should be closely 
related to the mirror map relating the moduli space 
of complex structures of a sigma model $Z$, and the
moduli space of complexified K\"ahler structures 
of the mirror sigma model $Z'$.  In the case when
both our Landau-Ginzburg models correspond to sigma 
models, then $\mathrm{coker}(B)_{\Bbb{C}/\Bbb{Z}}$
should map to the moduli $\mathcal{M}_K$ of complexified K\"ahler 
structures of $Z$ and $\mathrm{coker}(A')_{\Bbb{C}^\times}$
to the moduli, $\mathcal{M}_{L'}$ of complex structures of $Z'$.
This should lead to a commutative diagram
\begin{equation}
\begin{array}{ccc}
\mathrm{coker}(B)_{\Bbb{C}/\Bbb{Z}} & \stackrel{F_B}{\rightarrow} & \mathrm{coker}(A')_{\Bbb{C}^\times} \\
\downarrow & & \downarrow \\
\mathcal{M}_K & \stackrel{\mathrm{mirror}}{\rightarrow} & \mathcal{M}_{L'}
\end{array} .
\end{equation}

Nailing down the exact maps $F_A$ and $F_B$ in general requires more than what is known about
the mirror map since we hope that homological mirror symmetry could be made to work for 
more than just Calabi-Yau manifolds.  It is the subject of further research to understand
mirror symmetry and the mirror map in this more general setting.

\newpage

\bibliographystyle{halpha}
\bibliography{new_long_duality}

\begin{thebibliography}{CdlOGP91}

\bibitem[Aud04]{toricAudin}
Mich{\`e}le Audin.
\newblock {\em Torus actions on symplectic manifolds}, volume~93 of {\em
  Progress in Mathematics}.
\newblock Birkh\"auser Verlag, Basel, revised edition, 2004.

\bibitem[Bat]{alg-geom/9711008}
Victor~V. Batyrev.
\newblock {Stringy Hodge numbers of varieties with Gorenstein canonical
  singularities}.
\newblock {\em arXiv:alg-geom/9711008}, arXiv:alg-geom/9711008.

\bibitem[Bat94]{dualBatyrev}
Victor~V. Batyrev.
\newblock Dual polyhedra and mirror symmetry for {C}alabi-{Y}au hypersurfaces
  in toric varieties.
\newblock {\em J. Algebraic Geom.}, 3(3):493--535, 1994,
  arXiv:alg-geom/9310003.

\bibitem[BB]{alg-geom/9509009}
Victor~V. Batyrev and Lev~A. Borisov.
\newblock {Mirror duality and string-theoretic Hodge numbers}.
\newblock {\em arXiv:alg-geom/9509009}.

\bibitem[BH92]{BerglundHubsch:genConstruct}
Per Berglund and Tristan H{\"u}bsch.
\newblock A generalized construction of mirror manifolds.
\newblock In {\em Essays on mirror manifolds}, pages 388--407. Int. Press, Hong
  Kong, 1992, arXiv:hep-th/9201014.

\bibitem[Bor93]{dualNefBorisov}
Lev Borisov.
\newblock {Towards the Mirror Symmetry for Calabi-Yau Complete intersections in
  Gorenstein Toric Fano Varieties}.
\newblock Available as eprint only, 1993, arXiv:alg-geom/9310001.

\bibitem[CdlOGP91]{CandelasDeLaOssaGreenParkes}
Philip Candelas, Xenia~C. de~la Ossa, Paul~S. Green, and Linda Parkes.
\newblock A pair of {C}alabi-{Y}au manifolds as an exactly soluble
  superconformal theory.
\newblock {\em Nuclear Phys. B}, 359(1):21--74, 1991.

\bibitem[Cox95]{CoxHomog}
David~A. Cox.
\newblock The homogeneous coordinate ring of a toric variety.
\newblock {\em J. Algebraic Geom.}, 4(1):17--50, 1995, arXiv:alg-geom/9210008.

\bibitem[Ful93]{toricFulton}
William Fulton.
\newblock {\em Introduction to toric varieties}, volume 131 of {\em Annals of
  Mathematics Studies}.
\newblock Princeton University Press, Princeton, NJ, 1993.
\newblock The William H. Roever Lectures in Geometry.

\bibitem[Giv98a]{ellipticGivental}
Alexander Givental.
\newblock Elliptic {G}romov-{W}itten invariants and the generalized mirror
  conjecture.
\newblock In {\em Integrable systems and algebraic geometry (Kobe/Kyoto,
  1997)}, pages 107--155. World Sci. Publishing, River Edge, NJ, 1998,
  http://math.berkeley.edu/\~ \ giventh/papers/ell.pdf.

\bibitem[Giv98b]{mirrorTheoremGivental}
Alexander Givental.
\newblock A mirror theorem for toric complete intersections.
\newblock In {\em Topologial field theory, primitive forms and related topics
  (Kyoto, 1996)}, volume 160 of {\em Progr. Math.}, pages 141--175.
  Birkh\"auser Boston, Boston, MA, 1998, http://math.berkeley.edu/\~ \
  giventh/papers/tmp.pdf.

\bibitem[GS08]{SharpeGuffin:A-twistLG}
Josh Guffin and Eric Sharpe.
\newblock {A-twisted Landau-Ginzburg models}.
\newblock {\em arXiv:0801.3836}, 2008.

\bibitem[HV00]{HoriVafa}
Kentaro Hori and Cumrun Vafa.
\newblock Mirror symmetry.
\newblock Available as eprint only, 2000, arXiv:hep-th/0002222.

\bibitem[Kon95a]{kontsevich-icm}
Maxim Kontsevich.
\newblock Homological algebra of mirror symmetry.
\newblock In {\em Proceedings of the International Congress of Mathematicians,
  Vol.\ 1, 2 (Z\"urich, 1994)}, pages 120--139, Basel, 1995. Birkh\"auser,
  arXiv:alg-geom/9411018.

\bibitem[Kon95b]{kontsevich-motivicInt}
Maxim Kontsevich.
\newblock {Lecture at Orsay}.
\newblock 1995.

\bibitem[Oda88]{toricOda}
Tadao Oda.
\newblock {\em Convex bodies and algebraic geometry}, volume~15 of {\em
  Ergebnisse der Mathematik und ihrer Grenzgebiete (3) [Results in Mathematics
  and Related Areas (3)]}.
\newblock Springer-Verlag, Berlin, 1988.
\newblock An introduction to the theory of toric varieties, Translated from the
  Japanese.

\bibitem[Orl04]{orlov-singDLG}
D.~O. Orlov.
\newblock Triangulated categories of singularities and {D}-branes in
  {L}andau-{G}inzburg models.
\newblock {\em Tr. Mat. Inst. Steklova}, 246(Algebr. Geom. Metody, Svyazi i
  Prilozh.):240--262, 2004.

\bibitem[Sei01]{seidel-vanMut}
Paul Seidel.
\newblock Vanishing cycles and mutation.
\newblock In {\em European Congress of Mathematics, Vol. II (Barcelona, 2000)},
  volume 202 of {\em Progr. Math.}, pages 65--85. Birkh\"auser, Basel, 2001.

\bibitem[Sha]{Sharpe:private}
Eric Sharpe.
\newblock Private communication.

\bibitem[Sum74]{sumihiro1}
H.~Sumihiro.
\newblock {Equivariant Completion I}.
\newblock {\em J. Math Kyoto Univ.}, 14:1--28, 1974.

\bibitem[Sum75]{sumihiro2}
H.~Sumihiro.
\newblock {Equivariant Completion II}.
\newblock {\em J. Math Kyoto Univ.}, 15:573--605, 1975.

\bibitem[Wit93]{phasesWitten}
Edward Witten.
\newblock Phases of n = 2 theories in two dimensions.
\newblock {\em Nucl. Phys.}, B403:159--222, 1993, arXiv:hep-th/9301042.

\end{thebibliography}

\end{document}